\documentclass[a4paper,11pt]{amsart}
 \usepackage{float}
\usepackage{amssymb}
\usepackage{mathrsfs}
\usepackage{amsmath,amssymb,amsthm,latexsym,amscd,mathrsfs}
\usepackage{indentfirst}
\usepackage{stmaryrd}
\usepackage{graphicx}
\usepackage{extarrows}
\usepackage{multirow}
\usepackage{latexsym}
\usepackage{amsfonts}
\usepackage{color}
\usepackage{pictexwd,dcpic}
\usepackage{graphicx}
\usepackage{psfrag}
\usepackage{hyperref}
\usepackage{comment}
\usepackage{enumerate}

\usepackage{setspace}
\numberwithin{equation}{section} \theoremstyle{plain}
\newtheorem{thm}{Theorem}[section]

\newtheorem{prop}[thm]{Proposition}
\newtheorem{lem}[thm]{Lemma}
\newtheorem{cor}[thm]{Corollary}
\newtheorem{defn}[thm]{Definition}
\newtheorem{conj}[thm]{Conjecture}

\newtheorem{rem}[thm]{Remark}

\newtheorem*{acknow}{Acknowledgments}
   
 \makeatletter

\def\<{\langle}
\def\>{\rangle}
\def\({\left(}
\def\){\right)}
\def\[{\left[}
\def\]{\right]}
\def\tr{\mathop{\text{tr}}}

\makeatother
\title{Volume gap for minimal submanifolds in spheres}

\author[J.Q. Ge]{Jianquan Ge${}^{1}$}
\address{$^{1}$School of Mathematical Sciences, Laboratory of Mathematics and Complex Systems, Beijing Normal University, Beijing 100875, P.R. CHINA.}
\email{jqge@bnu.edu.cn}

\author[F.G. Li]{Fagui Li${}^{2,*}$}
\address{$^{2}$Frontier Interdisciplinary Domain, Beijing Institute of Technology, Zhuhai, Guangdong 519088, P. R. CHINA.}
\email{lifagui@bitzh.edu.cn}
\date{}
\thanks{$^{*}$ the corresponding author.}
\thanks{J. Q. Ge is partially supported by NSFC  (No.12171037, 12571049)  and the Fundamental Research Funds for the Central Universities.}
\thanks{F. G. Li is partially supported by  NSFC (No. 12271040, 12501061) and Research Start up Funding of Beijing Institute of Technology (No. 5640011253301).}
\begin{document}
\maketitle

% % % % % % % % % %
\begin{abstract}
For a closed minimal submanifold $f:M^n\looparrowright \mathbb{S}^{N}$ in the unit sphere $(n<N)$, we prove
$$
{\rm Vol}(M^n) 
\geq
 \frac{n+1}{n+2}
\int_{M}\left( 1+\varphi_{p}^2\right) 
\geq 
m{\rm Vol}(\mathbb{S}^{n}),$$
where $\varphi_{p}(x):=\langle f(x),p\rangle$ is the height function in direction $p\in f(M)$, $m$ denotes the multiplicity of $p\in f(M)$ and ${\rm Vol}$  denotes the Riemannian volume functional, and each equality holds if and only if $M$ is totally geodesic. 
As an application, if the volume of $M^n$ is less than or equal to the volume of any $n$-dimensional minimal Clifford torus, then $M^n$ must be embedded, verifying the non-embedded case of Yau's conjecture. In addition, we also get volume gaps for minimal hypersurfaces with constant scalar curvature, improving Cheng-Li-Yau's classical volume gap in this case. Some other volume gaps and related pinching rigidities are also obtained.
\end{abstract}

\section{Introduction}
In  1984, Cheng-Li-Yau \cite{Cheng Li Yau 1984 Heat equations} proved  that \emph{if $M^n$ is a closed minimal immersed submanifold in the unit sphere $\mathbb{S}^{N}$ $\left( N=n+l\geq n+1\right) $ and $M^n$ is of maximal dimension
$($$M^n$ does not lie on any hyperplane of $\mathbb{R}^{N+1}$$)$,
then %there is  a positive constant $c(n,N)>0$, depending only on $n$ and $N$, such that
the volume  of  $M^n$ satisfies}
  $$
  {\rm Vol}(M^n)>\left( 1+\frac{2l+1}{B_n}\right) {\rm Vol}(\mathbb{S}^n),
  $$
where $B_n=2n+3+2\exp \left( 2nC_n\right)$,
$C_n= \frac{1}{2}n^{n/2}e\Gamma(n/2,1)$  and $\Gamma\left(\frac{n}{2},1\right)=\int_{1}^\infty e^{-t}t^{\frac{n}{2}-1}~dt$.	 
This means that there is a volume gap for minimal submanifolds in spheres. Regarding the second smallest volume,
Yau proposed the following well-known conjecture in his Problem Section \cite{Yau  1992}:
 \begin{conj}[\textbf{Yau's Conjecture \cite{Yau  1992}}] \label{conj  Yau 1982}
The volume of one of the minimal Clifford torus $(M_{k,n-k}=S^{k}(\sqrt{\frac{k}{n}})\times S^{n-k}(\sqrt{\frac{n-k}{n}}),\ 1\leq k\leq n-1)$  gives the lowest value of volume among all non-totally geodesic closed minimal hypersurfaces of $\mathbb{S}^{n+1}$.
 \end{conj}
 For $n=2$,
 Conjecture \ref{conj  Yau 1982}  (also called the Solomon-Yau conjecture \cite{Ge19,IW15}) is true, due to the following two results. On the one hand, Li and Yau \cite{Li and Yau 1982} proved that \emph{let $f: M^2\looparrowright    \mathbb{S}^{N} $ be a minimal immersion of a closed surface into the unit sphere  $\mathbb{S}^{N}$.
    If there exists a point such that its preimage set  consists of $m$ distinct points in $M^2$, then ${\rm Vol}(M^2)\geq
    m {\rm Vol}(\mathbb{S}^2)= 4m\pi$}.
    Recently, in a previous version of this paper \cite{Ge Li 2022 Solomon-Yau conj} we partially generalized this inequality to any dimension which is sufficient to verify Yau's conjecture in the non-embedded case.
     %for $n\geq 3$,
 %    i.e.,   
 %   $
 %   {\rm Vol}(M^n)
 %  \geq
  %   \frac{m}{2}{\rm Vol}(\mathbb{S}^{n})+m\frac{\sqrt{n+1}}{n}
 %  {\rm Vol}(\mathbb{S}^{n-1}).$ 
    Nguyen \cite{Nguyen  2023} proved $
        {\rm Vol}(M^n)
        \geq
        m{\rm Vol}(\mathbb{S}^{n})$ finally. 
     More generally,  Li and Yau \cite{Li and Yau 1982} proved both of Yau's conjecture and Willmore's conjecture for non-embedded case in dimension $2$.
On the other hand, Calabi \cite{Calabi 1967} proved that \emph{if ${S}^{2}$ is minimally immersed in $\mathbb{S}^{3}$, then ${S}^{2}$ is an equator (i.e,  totally geodesic).}
%it has been shown by
 Marques and Neves \cite{Marques and Neves 2014} completely proved  Willmore's conjecture and  showed that \emph{any non-totally geodesic closed minimal embedded surface in $\mathbb{S}^3$ has volume greater than or equal to $2\pi^2$, the volume of the Clifford torus $M_{1,1}$. The equality holds only for the Clifford torus $M_{1,1}$.} This completes the proof for embedded case in dimension $2$.
 Another related rigidity result is Lawson's conjecture (proved by Brendle \cite{Brendle S 2013}), i.e., \emph{the only embedded minimal torus in $\mathbb{S}^3$ is the Clifford torus.}
 For more details of minimal surfaces,  please see   \cite{Andrews Ben Li Haizhong 2015,Brendle S 2013 survey of recent results,Osserman Robert 1980}, etc.

  For $n\geq 3$, among minimal rotational hypersurfaces, Yau's conjecture was verified for $2\leq n\leq 100$ by Perdomo and Wei \cite{Perdomo and  Wei 2015}, and for all dimensions by  Cheng-Wei-Zeng \cite{Cheng Qing Ming Guoxin Wei and Yuting Zeng 2019} and Cheng-Lai-Wei \cite{Cheng Qing Ming  Junqi Lai and Guoxin Wei 2024}, respectively. Remarkably, in the asymptotic sense, Ilmanen and White \cite{IW15} proved %the Solomon-Yau Conjecture
 Yau's conjecture  in the class of \textit{topologically nontrivial} (at least one of the components of $\mathbb{S}^{n+1}\setminus M^n$ is not contractible) closed minimal embedded hypersurfaces whose hypercones are area-minimizing in $\mathbb{R}^{n+2}$. Namely, they showed $ {\rm Vol}(M^n)>\sqrt{2}{\rm Vol}(\mathbb{S}^n)$, where $\sqrt{2}=\lim\limits_{k\rightarrow\infty}\frac{{\rm Vol}(M_{k,k})}{{\rm Vol}(\mathbb{S}^n)}=\lim\limits_{k\rightarrow\infty}\frac{{\rm Vol}(M_{k,k+1})}{{\rm Vol}(\mathbb{S}^n)}$ for $n=2k$ or $n=2k+1$ respectively.
Recently, Viana \cite{Celso Viana 2023}
confirmed Yau's conjecture in the class of minimal hypersurfaces separating $\mathbb{S}^n$ in two connected regions which are both antipodal invariant.

 In this paper,  we study the  monotonicity formula of   Choe and  Gulliver \cite{Choe and Gulliver 1992} (see Proposition \ref{prop The monotonicity formula for minimal submanifolds}) for any minimal submanifold in $\mathbb{S}^{N}$ similar to Euclidean space. For the classical monotonicity formula in Euclidean space, please see the excellent and elegant  survey  by Brendle \cite{Brendle S 2020}. As  applications,  we prove  the following results. The first one shows that the volume lower bound $m{\rm Vol}(\mathbb{S}^{n})$ would not be attained once the minimal submanifold $M^n$ is not totally geodesic.
\begin{thm}\label{theorem  gap  results of volume in the m preimage varphi p2}
Let $f: M^n\looparrowright    \mathbb{S}^{N} $ be a   closed minimal  immersed submanifold  in  $\mathbb{S}^{N}$. If there exists a point $p\in f(M)$ such that its preimage set  consists of $m$ distinct points in $M^n$, then
$$
{\rm Vol}(M^n) 
\geq
 \frac{n+1}{n+2}
\int_{M}\left( 1+\varphi_{p}^2\right) 
\geq 
m{\rm Vol}(\mathbb{S}^{n}),$$
where each equality holds if and only if $M^n$ is totally geodesic and   $\varphi_{p}$ is the height function (see (\ref{height functions varphi})).
\end{thm}

%\begin{rem}
% Manh Tien Nguyen \cite claimed that ${\rm Vol}(M^n)\geq m{\rm Vol}(\mathbb{S}^{n})$ is true without any conditions in Theorem \ref{theorem  gap  results of volume in the introduction}.
%\end{rem}

\begin{cor}\label{thm 2 gap  results of volume in the introduction must be embedded}
 Let $M^n$ be a closed minimal  immersed submanifold  in  $\mathbb{S}^{N}$. If
 $$
 {\rm Vol}(M^n)
<
2{\rm Vol}(\mathbb{S}^{n}),$$
then $M^n$ must be embedded.
\end{cor}
%\begin{thm}\label{thm 2 gap  results of volume in the introduction}
% Let $M^n$ be a closed minimal  immersed submanifold  in  $\mathbb{S}^{N}$. If
% $$
% {\rm Vol}(M^n)
%\leq
%\max_{1\leq k\leq n-1} {\rm Vol}(M_{k,n-k}),$$
%then $M^n$ must be embedded.
%  \end{thm}
\begin{rem}
Corollary \ref{thm 2 gap  results of volume in the introduction must be embedded} shows that  Yau's conjecture  is correct for non-embedded hypersurfaces in $\mathbb{S}^{n+1}$.
\end{rem}
  \begin{cor}\label{theorem  gap  results of volume nonembedded hypersurface}
  Let $f: M^n\looparrowright    \mathbb{S}^{n+1} $ be a    closed minimal  immersed hypersurface in  $\mathbb{S}^{n+1}$. If there exists a point $p\in f(M)$ such that its preimage set  consists of $m$ distinct points in $M^n$, then
\begin{equation*} 
  {\rm Vol}(M^n) 
  \geq 
  m{\rm Vol}(\mathbb{S}^{n})+
   \frac{1}{n+2}
   \int_{M} \psi_{p}^2,
\end{equation*}    
  where the equality holds if and only if $M^n$ is totally geodesic and  $\psi_{p}$ is the normal  height function (see (\ref{height functions psi})).
  \end{cor}
Let $S=\|A\|^2$ denote the squared length of second fundamental form $A$  of $M^n$,  $ S_{\max}=\sup_{p\in M^n}S(p)$ and $ S_{\min}=\inf_{p\in M^n}S(p)$, we have
  \begin{cor}\label{theorem  gap  results of volume nonembedded IE n dim to n+1 sphere}
  Let $f: M^n\looparrowright    \mathbb{S}^{n+1} $ be a   closed, non-totally-geodesic, minimal Integral-Einstein (see Definition  \ref{defintion integral Eins manifold})  immersed  hypersurface  in  $\mathbb{S}^{n+1}$. If there exists a point $p\in f(M)$ such that its preimage set  consists of $m$ distinct points in $M^n$, then
%   \begin{equation*}
%  \int_M \psi_a^2\geq \frac{S_{\min}}{ (n+1)S_{\max}+ S_{\min}}{{\rm Vol}(M^n)}.
%   \end{equation*}
  $$
  {\rm Vol}(M^n) 
  \geq 
  \frac{(n+1)(n+2)S_{\max}+(n+2)S_{\min}}{(n+1)(n+2)S_{\max}+(n+1)S_{\min}}
m
  {\rm Vol}(\mathbb{S}^{n}).$$
%  where $ S_{\max}=\sup_{p\in M^n}S(p)$ and $ S_{\min}=\inf_{p\in M^n}S(p)$.
  \end{cor}
  In particular, for $n=2$, we have
  \begin{cor}\label{theorem  gap  results of volume nonembedded 2 dim to 3 sphere}
  Let $f: M^2\looparrowright    \mathbb{S}^{3} $ be a   closed, non-totally-geodesic, minimal  immersed  surface  in  $\mathbb{S}^{3}$. If there exists a point $p\in f(M)$ such that its preimage set  consists of $m$ distinct points in $M^n$, then
  $$
  {\rm Vol}(M^2) 
  \geq 
  \frac{12S_{\max}+4S_{\min}}{12S_{\max}+3S_{\min}}
  m{\rm Vol}(\mathbb{S}^{2}).$$
%  where $ S_{\max}=\sup_{p\in M^2}S(p)$ and $ S_{\min}=\inf_{p\in M^2}S(p)$.
  \end{cor}
%The preceding volume gaps make sense only for non-embedded submanifolds.
 In the following we give volume gaps that fit for hypersurfaces with some curvature conditions as in the famous Chern conjecture (cf. \cite{Chang 1993,Peng and Terng1 1983,M Scherfner S Weiss and  Yau 2012,TY20}, etc.), which is still open for $n>3$ and states that \emph{a closed  minimal immersed hypersurface  in $\mathbb{S}^{n+1}$ with constant scalar curvature  is isoparametric.}
%The preceding volume gaps make sense only for non-embedded submanifolds. 
%In the following we give volume gaps for minimal hypersurfaces  with  some curvature conditions  as in the  Chern Conjecture.
%The preceding volume gaps make sense only for non-embedded submanifolds. In the following we give volume gaps that fit for embedded hypersurfaces, under some curvature conditions as in the famous Chern Conjecture (cf. \cite{Chang 1993,Peng and Terng1 1983,M Scherfner S Weiss and  Yau 2012,TY20}, etc.), which is still open for $n>3$ and states that \emph{a closed  minimal immersed hypersurface  in $\mathbb{S}^{n+1}$ with constant scalar curvature  is isoparametric.}

\begin{thm}\label{theorem  gap  results of volume for hypersurface constant S}
 Let $M^n$ be a closed, non-totally-geodesic, minimal immersed  hypersurface  with constant scalar curvature  in  $\mathbb{S}^{n+1}$. Then
   $$
   {\rm Vol}(M^n)
  \geq
  \frac{2n(n+2)}{2n^2+4n-1}
   {\rm Vol}(\mathbb{S}^{n}).$$
    \end{thm}

\begin{rem}
This improves the volume gap of Cheng-Li-Yau \cite{Cheng Li Yau 1984 Heat equations} under the additional assumption of constant scalar curvature. In fact, Theorem \ref{theorem  gap  results of volume for hypersurface constant S} is a special case of Theorem \ref{theorem  gap  results of volume for hypersurface  non-constant S}, where the case of non-constant scalar curvature is also considered. As an application, a new pinching rigidity result is obtained in Corollary \ref{corollary gap  results of volume for hypersurface  non-constant S new}.
\end{rem}

\begin{thm}\label{theorem  gap  results of volume for hypersurface  constant S f3 embedded}
 Let $M^n$ be a closed, non-totally-geodesic, minimal immersed  hypersurface  in  $\mathbb{S}^{n+1}$ with constant scalar curvature. If the third mean curvature is constant (or $M^n$ is Integral-Einstein, see Definition  \ref{defintion integral Eins manifold}) and  there exists a point $p\in f(M)$ such that its preimage set  consists of $m$ distinct points in $M^n$, then
 $$
 {\rm Vol}(M^n)
\geq 
\frac{m(n+2)^2}{(n+1)(n+3)}{\rm Vol}(\mathbb{S}^{n})>m{\rm Vol}(\mathbb{S}^{n}).$$
  \end{thm}

The rest of the paper is organized as follows. In section \ref{sect2} we introduce the monotonicity formula for minimal submanifolds in spheres and some preliminary results.
In section \ref{sect3} we prove Theorem \ref{theorem  gap  results of volume in the m preimage varphi p2} and its corollaries. Section \ref{sect4} presents volume gaps for minimal hypersurfaces with additional curvature assumptions (Theorems \ref{theorem  gap  results of volume for hypersurface constant S}, \ref{theorem  gap  results of volume for hypersurface  constant S f3 embedded}, \ref{theorem  gap  results of volume for hypersurface  constant S f3}, \ref{theorem  gap  results of volume for hypersurface  non-constant S antipodal map}, \ref{theorem  gap  results of volume for hypersurface  non-constant S} and Corollaries \ref{corollary gap  results of volume for hypersurface  non-constant SUX  antipodal map}, \ref{corollary gap  results of volume for hypersurface  non-constant SUX}), and give applications to pinching rigidity results (Corollaries \ref{corollary gap  results of volume for hypersurface  non-constant S antipodal}, \ref{corollary gap  results of volume for hypersurface  non-constant S new}). In the last section \ref{sect5}, we 
study volume difference for the two parts of minimal submanifold in the two hemispheres divided by an equator (Theorem \ref{theorem Volume difference} and Corollaries \ref{cor Volume gap of minimal hypersurface by Volume difference}, \ref{cor Volume of minimal hypersurface partitioning}, \ref{cor Volume of minimal S constant IE hypersurface partitioning}).

%In fact, Theorem \ref{theorem  gap  results of volume for hypersurface constant S and antipodal map} is a special case of Theorem \ref{theorem  gap  results of volume for hypersurface  non-constant S}, where minimal hypersurfaces with non-constant scalar curvature are also considered. As an application, a new pinching rigidity result is obtained in Corollary \ref{corollary gap  results of volume for hypersurface  non-constant S antipodal}.

\section{The monotonicity formula for minimal submanifolds in spheres}\label{sect2}
Let $f: M^n\looparrowright    \mathbb{S}^{N} \subset \mathbb{R}^{N+1}$ be a closed minimal immersed submanifold  in the unit sphere  $\mathbb{S}^{N}$. For any fixed unit vector $a \in \mathbb{S}^{N}$, we consider the height function on $M^n$,
\begin{equation}\label{height functions varphi}
\varphi_a(x) = \langle f(x),a \rangle.
\end{equation}
Then we have the following basic properties.
\begin{prop} \cite{Ge Li 2020, Takahashi 1966} \label{prop funda}
For all $a\in \mathbb{S}^{N}$, we have $-1\leq \varphi_a\leq 1$, and
$$\begin{array}{lll}
\nabla \varphi_a=a^{\rm T},&
\Delta \varphi_a=-n\varphi_a,& \int_M \varphi_a=0,
\end{array}$$
 where  $a^{\rm T}\in \Gamma(TM)$ denotes the tangent component of $a$ along $M^n$.
\end{prop}
%Let $f: M^n\looparrowright    \mathbb{S}^{N} \subset \mathbb{R}^{N+1}$ be a closed minimal immersed submanifold  in the unit sphere  $\mathbb{S}^{N}$. %For any fixed unit vector $a \in \mathbb{S}^{N}$, we consider the height function on $M^n$,
%\begin{equation}\label{height functions varphi}
%\varphi_a(x) = \langle f(x),a \rangle.
%\end{equation}
In particular, if $M^n$ is a minimal hypersurface in   $\mathbb{S}^{n+1}$ and denote by $\nu(x)$ the (local) unit normal vector field of $M^n$, the (local) normal height function   are defined as
\begin{equation}\label{height functions psi}
  \psi_a (x)=\langle \nu,a \rangle
\end{equation}
for any fixed unit vector $a \in \mathbb{S}^{N}$.  There are many applications
of minimal hypersurfaces in spheres by using the height functions (cf. \cite{Ge Li 2020,Ge Li 2021 Perdomo conjecture}, etc.), recently.
%Then we have the following basic properties.

\begin{prop} \cite{Ge Li 2020,  Nomizu and Smyth 1969} \label{prop funda psi}
For all $a\in \mathbb{S}^{n+1}$, we have
$$\begin{array}{lll}
\nabla \psi_a =-A(a^{\rm T}), & \Delta \psi_a = -S\psi_a,
\end{array}$$
 where  $A$ is the shape operator with respect to $\nu$ (i.e., $A(X)=-\overline\nabla_X\nu$) and  $S=\|A\|^2={\tr} \left( AA^t\right) $.
\end{prop}
% In the following monotonicity formula for minimal submanifolds, our notations are different from that of Choe and  Gulliver \cite{Choe and Gulliver 1992}, therefore we give here a simple proof. Denote the level sets by $$ {\left\lbrace \varphi_a= t\right\rbrace }={\left \lbrace x\in M^n: \varphi_a(x)= t\right\rbrace },\ \  {\left\lbrace s\leq \varphi_a\leq t\right\rbrace }={\left \lbrace x\in M^n: s\leq  \varphi_a(x)\leq t\right\rbrace }. $$

 In the following monotonicity formula for minimal submanifolds, our notations are different from that of Choe and  Gulliver \cite{Choe and Gulliver 1992}, therefore we give here a simple proof. Denote the level sets by $$ {\left\lbrace \varphi_a= t\right\rbrace }={\left \lbrace x\in M^n: \varphi_a(x)= t\right\rbrace },\ \  {\left\lbrace s\leq \varphi_a\leq t\right\rbrace }={\left \lbrace x\in M^n: s\leq  \varphi_a(x)\leq t\right\rbrace }.
$$
In the subsequent proof, we need the following results.
\begin{prop}\label{prop The monotonicity formula for minimal submanifolds}\cite{Choe and Gulliver 1992}
For any fixed unit vector $a \in \mathbb{S}^{N}$, if $M\not\subset \left\lbrace\varphi_a=0\right\rbrace$, then the function
$$
r\longmapsto \frac{ \int_{\left\lbrace \varphi_a\geq r\right\rbrace }
\varphi_a}{\left(1-r^2\right)^{\frac{n}{2}} }$$
is
 monotone increasing for $-1<r\leq 0$ and
 monotone decreasing for $0<r<1$.
 \end{prop}
\begin{proof}
%{\bf Case (i).} Let's deal with $0< t<1$.
By Proposition \ref{prop funda}, $$\nabla \varphi_a=a^{\rm T},\ \ \Delta \varphi_a=-n\varphi_a.$$
Due to  the divergence theorem and
\begin{equation}\label{equation a2 psi2 vaephi2}
|a^{\rm T}|^2+\varphi_a^2\leq 1,
\end{equation}
one has for all $-1<t<1$,
\begin{equation}\label{equation Stokess formula to varphia}
\int_{\left\lbrace \varphi_a\geq t\right\rbrace }
\varphi_a=
\int_{\left\lbrace \varphi_a= t\right\rbrace}\frac{|a^{\rm T}|}{n}\leq
\int_{\left\lbrace \varphi_a= t\right\rbrace}\frac{\sqrt{1-\varphi_a^2}}{n}=
\int_{\left\lbrace \varphi_a= t\right\rbrace}\frac{\sqrt{1-t^2}}{n}.
\end{equation}
For all $0 < s\leq r< 1$, by the    co-area formula and (\ref{equation a2 psi2 vaephi2}, \ref{equation Stokess formula to varphia}), we obtain
\begin{equation}\label{equation 1 leqr to geqr varphia}
\begin{aligned}
\int_{\left\lbrace s\leq\varphi_a \leq r\right\rbrace }
\varphi_a
&=
\int_s^r dt
\int_{\left\lbrace \varphi_a=t\right\rbrace }
\frac{\varphi_a}{|a^{\rm T}|}
\geq
\int_s^r dt
\int_{\left\lbrace \varphi_a=t\right\rbrace }
\frac{\varphi_a}{\sqrt{1-\varphi_a^2}}\\
&=
\int_s^r dt
\int_{\left\lbrace \varphi_a=t\right\rbrace }
\frac{t}{\sqrt{1-t^2}}\\
%\geq
%\int_s^r
%\int_{\left\lbrace |\varphi_a|\geq t\right\rbrace }
%\frac{t}{\sqrt{1-t^2}}
%\frac{n}{\sqrt{1-t^2}}
%|\varphi_a|\\
&\geq
\int_s^r dt
\int_{\left\lbrace \varphi_a\geq t\right\rbrace }
\frac{nt}{{1-t^2}}
\varphi_a.
\end{aligned}
\end{equation}
Thus, by (\ref{equation 1 leqr to geqr varphia})
\begin{equation}\label{inequation varpsi sr geq}
\frac{d\int_{\left\lbrace 0\leq \varphi_a \leq r\right\rbrace }
\varphi_a}
{dr}
=\lim_{s\to {r}}
\frac{\int_{\left\lbrace 0\leq \varphi_a \leq r\right\rbrace }
\varphi_a-\int_{\left\lbrace 0\leq \varphi_a \leq s\right\rbrace }
\varphi_a}
{r-s}
\geq
\frac{nr}{{1-r^2}}
\int_{\left\lbrace \varphi_a\geq r\right\rbrace }
\varphi_a.
\end{equation}
Since (\ref{inequation varpsi sr geq}) and
\begin{equation*}
\int_{\left\lbrace \varphi_a\geq 0\right\rbrace  }
\varphi_a=
\int_{\left\lbrace 0 \leq \varphi_a \leq r\right\rbrace }
\varphi_a+\int_{\left\lbrace \varphi_a \geq r\right\rbrace }
\varphi_a,
\end{equation*}
we have
\begin{equation*}
\frac{d \mathrm{In} \int_{\left\lbrace \varphi_a \geq r\right\rbrace }
\varphi_a}
{dr}\leq -
\frac{nr}{{1-r^2}}.
\end{equation*}
Hence, for $0< s_1\leq s_2< 1$, one has
\begin{equation*}
\int_{s_1}^{s_2}
d\mathrm{In} \int_{\left\lbrace \varphi_a \geq r\right\rbrace }
\varphi_a
\leq -
\int_{s_1}^{s_2}\frac{nr}{{1-r^2}}dr=
\frac{n}{2}\mathrm{ln}\left( \frac{1-s_2^2}{1-s_1^2}\right),
\end{equation*}
and
$$
\frac{ \int_{\left\lbrace \varphi_a\geq s_1\right\rbrace }
\varphi_a}{\left(1-s_1^2\right)^{\frac{n}{2}} }
\geq
\frac{ \int_{\left\lbrace \varphi_a\geq s_2\right\rbrace }
\varphi_a}{\left(1-s_2^2\right)^{\frac{n}{2}} }.
$$
This shows the monotonicity for $0<r<1$.

For $-1<r\leq 0$, we just need to change $\varphi_a$ to $-\varphi_a$. Similarly by the     co-area formula and (\ref{equation a2 psi2 vaephi2}, \ref{equation Stokess formula to varphia}), %one has
%\begin{equation*}\label{equation Stokess formula to varphia -1 to 0}
%\int_{\left\lbrace \varphi_a\geq t\right\rbrace }
%\varphi_a=
%\int_{\left\lbrace \varphi_a= t\right\rbrace}\frac{|a^{\rm T}|}{n}\leq
%\int_{\left\lbrace \varphi_a= t\right\rbrace}\frac{\sqrt{1-t^2}}{n},
%\end{equation*} and
 for all $-1< s\leq r< 0$, we obtain
\begin{equation*}\label{equation 1 leqr to geqr varphia -1 to 0}
\begin{aligned}
\int_{\left\lbrace s\leq\varphi_a \leq r\right\rbrace }
-\varphi_a
&=
\int_s^r dt
\int_{\left\lbrace \varphi_a=t\right\rbrace }
\frac{-\varphi_a}{|a^{\rm T}|}
&\geq
\int_s^r dt
\int_{\left\lbrace \varphi_a\geq t\right\rbrace }
\frac{-nt}{{1-t^2}}
\varphi_a.
\end{aligned}
\end{equation*}
Hence
\begin{equation*}
\frac{d\mathrm{In} \int_{\left\lbrace \varphi_a \geq r\right\rbrace }
\varphi_a}
{dr}\geq -
\frac{nr}{{1-r^2}},
\end{equation*}
where $\int_{\left\lbrace \varphi_a \geq r\right\rbrace }\varphi_a=-\int_{\left\lbrace \varphi_a \leq r\right\rbrace }\varphi_a>0$ for $r>\min \varphi_a$.
Thus, by integrating the above inequality  we have the monotonicity:  for $-1< s_1\leq s_2< 0$,
$$
\frac{ \int_{\left\lbrace \varphi_a\geq s_1\right\rbrace }
\varphi_a}{\left(1-s_1^2\right)^{\frac{n}{2}} }
\leq
\frac{ \int_{\left\lbrace \varphi_a\geq s_2\right\rbrace }
\varphi_a}{\left(1-s_2^2\right)^{\frac{n}{2}} }.
$$
\end{proof}
%\subsection {Notation and Preliminaries}
%\section{The monotonicity formula for minimal submanifolds in spheres}

\begin{lem}\label{lem  f1f2V} \cite{Nguyen  2023}
 Let $M^n$ be a  closed minimal immersed  submanifold  in  $\mathbb{S}^{N}$.  Suppose 
 $f_1, f_2$ and  $V$ are smooth  functions on $\mathbb{R}$, $ w_{k}={\rm Vol}(\mathbb{S}^{k}) $  and 
 $$
 	\theta_1(t):=\frac{A_1(t)}{B_1(t)},\  \theta_2(t):=\frac{A_2(t)}{B_2(t)},
 $$
 where
 $$
 \left\{\begin{array} { l } 
 	{ A _ { 1 } ( t ) : = \int _ {\left\lbrace \varphi_a\geq t\right\rbrace } f _ { 1 } ( \varphi _ { a } ) } \\
 	{ B _ { 1 } ( t ) : = w _ { n- 1 } \int _ { t } ^ { 1 } f _ { 1 } ( s ) V ^ { \frac { n } { 2 } - 1 } ( s ) d s }
 \end{array}
  \quad
   \left\{\begin{array}{l}
 	A_2(t):=\int_{\left\lbrace \varphi_a\geq t\right\rbrace } f_2\left(\varphi_a\right) \\
 	B_2(t):=w_{n-1} \int_t^1 f_2(s) V^{\frac{n}{2}-1}(s) d s
 \end{array}\right.\right.,
 $$
 then 
 $$
 \frac{d}{d t}\left(B_1\left(\theta_1-\theta_2\right)\right)=f_1 \cdot \frac{d \theta_2}{d t} \cdot\left(\frac{B_2}{f_2}-\frac{B_1}{f_1}\right).
 $$
    \end{lem}
%$$
%\begin{aligned}
%	& (v\geq 0, t\geq 0) \\
%	& \left(w_{k-1}=vol(S^{k-1})\right)
%\end{aligned}
%$$
\begin{proof}
By    co-area formula, we have
$$
\frac{d A_1}{d t} \frac{1}{f_1}=\frac{d A_2}{d t} \frac{1}{f_2} \quad\left(=-\int_{\{ \varphi_a=t\} } \frac{1}{\left|\nabla \varphi_a\right|}\right).
$$
By $A_1=\theta_1 B_1$ and $A_2=\theta_2 B_2$, one has
$$
	\frac{d\left(\theta_1 B_1\right)}{d t} \frac{1}{f_1}=\frac{d\left(\theta_2 B_2\right)}{d t} \frac{1}{f_2},
	$$
	and
		$$
	\frac{1}{f_1} \frac{d}{d t}\left(B_1\left(\theta_1-\theta_2\right)\right)+\frac{1}{f_1} \frac{d}{d t}\left(B_1 \theta_2\right)=\frac{1}{f_2}\left(B_2 \frac{d \theta_2}{d t}+\theta_2 \frac{d B_2}{d t}\right) .
		$$
		Hence
				$$
	\frac{1}{f_1} \frac{d}{d t}\left(B_1\left(\theta_1-\theta_2\right)\right)=\frac{d \theta_2}{d t}\left(\frac{B_2}{f_2}-\frac{B_1}{f_1}\right)+\frac{\theta_2}{f_2} \frac{d B_2}{d t}-\frac{\theta_2}{f_1} \frac{d B_1}{d t}.
	$$
Due to $\frac{d B_2}{d t} \cdot \frac{1}{f_2}=\frac{d B_1}{d t} \cdot \frac{1}{f_1}=-w_{n-1} V^{\frac{n}{2}-1}$, one has
$$
\frac{d}{d t}\left(B_1\left(\theta_1-\theta_2\right)\right)=f_1 \cdot \frac{d \theta_2}{d t} \cdot\left(\frac{B_2}{f_2}-\frac{B_1}{f_1}\right).
$$
\end{proof}
\begin{defn}\label{defintion integral Eins manifold}\cite{Ge Li 2020}
 Let $M^n$ $(n\geq3)$ be a compact submanifold in the Euclidean space $\mathbb{R}^{N}$. We call $M^n$ an Integral-Einstein (IE) submanifold if for any unit vector $a\in \mathbb{S}^{N-1}$,
 \begin{equation*}\label{equation integral Einstein manifold}
 \int_{M}\left( {\rm Ric}-\frac{R}{n}\mathbf{g}\right)
 (a^{\rm T},a^{\rm T})=0,
 \end{equation*}
  where ${\rm Ric}$, $R$ are the Ricci, scalar curvature and $\mathbf{g}$ is the induced metric of $M^n$.
 \end{defn}
\begin{thm}\label{thm-IE-hypers-charact}\cite{Ge Li 2020}
 Let $M^n$ be a closed  hypersurface immersed in $\mathbb{S}^{n+1}$. Then \emph{$M^n$} is Integral-Einstein (IE) if and only if for all  $a \in \mathbb{S}^{n+1}$,
 \begin{equation*}\label{equation integral Einstein  hypersurface}
 \int_{M}\left( 1-(n+1)\varphi_a^2-\psi_a^2\right)  =
 \int_{M}\left( \rho-1\right)\left( 1-\varphi_a^2-(n+1)\psi_a^2\right),
 \end{equation*}
 where $\rho-1=\frac{n^2H^2-S}{n(n-1)}$ and $\rho=\frac{R}{n(n-1)}$ is the normalized scalar curvature.
 In particular, we have  the following special cases.
 \begin{itemize}
 \item[(A)] If $M^n$ is minimal, then \emph{$M^n$} is IE if and only if
 $$
  \int_{M}S
 \left(  1-\varphi_a^2-(n+1)\psi_a^2\right) =0, \quad \textit{for all } a \in \mathbb{S}^{n+1}.
 $$
 \item[(B)]
 If $M^n$ is minimal and $S>0$ is constant, then we have
 \begin{equation*}\label{IEintegral-MCSC}
 \int_{M}\left( {\rm Ric}-\frac{R}{n}\mathbf{g}\right)
 (a^{\rm T},a^{\rm T})=S\Big((n+2) \int_{M}{\varphi_a^2}- {\rm Vol}(M^n)\Big).
 \end{equation*}
 In this case, \emph{$M^n$} is IE if and only if any one of the follows holds:
 \begin{itemize}
 \item
$
 \int_{M}\varphi_a^2=\frac{1}{n+2}{\rm Vol }(M^n)$ for all  $a \in \mathbb{S}^{n+1};
$
 \item
 $
 \int_{M}\psi_a^2=\frac{1}{n+2}{\rm Vol }(M^n)$ for all  $ a \in \mathbb{S}^{n+1};
 $
 \item
 $
 \int_{M}\varphi_a^2=\int_{M}\psi_a^2 $ for all  $a \in \mathbb{S}^{n+1};
 $
 \item
 $
 \int_{M}\varphi_a \psi_a f_3=0 $ for all  $ a \in \mathbb{S}^{n+1},
 $
  where $f_3={\rm Tr}(A^3)=3\binom{n}{3}H_3$ and $H_3$ is the third mean curvature.
 \end{itemize}
 \end{itemize}
 \end{thm}
% \begin{cor}\label{cor-f3-isop}\cite{Ge Li 2020}
% A closed minimal hypersurface of constant scalar curvature in $\mathbb{S}^{n+1}$  with $S>n$ and constant third mean curvature is an IE hypersurface.
% In particular, every minimal isoparametric hypersurface with $g\geq3$ principal curvatures in $\mathbb{S}^{n+1}$ is an IE hypersurface. Moreover, the Clifford torus $S^{1}(r)\times S^{n-1}(\sqrt{1-r^2})\subset\mathbb{S}^{n+1}$ $(0<r<1)$ is not IE.
% \end{cor}
 
 \begin{thm}\cite{Choe and Gulliver 1992}\label{thm Choe and Gulliver 1992 M and partial M}
 Let $M^n$ be a minimal submanifold with boundary $\partial M$ in $\mathbb{S}^{N}$ . Assume that $\partial M$ lies on a geodesic sphere of $\mathbb{S}^{N}$ with radius $r\leq  \frac{\pi}{2}$ and  centered at a point $p\in M^n\subset \mathbb{S}^{N}$.  Then
\begin{equation}
n^n{\rm Vol}(\mathbb{B}^{n}) 
 \left(
\int_{M} \varphi _ { p }\right) ^{n-1}\leq 
{\rm Vol}^n(\partial M),
\end{equation}
where the equality holds if and only if $M^n$ is a totally geodesic ball centered at $p$.
 \end{thm}
 \begin{prop}\label{prop smoothness of nodal set}\cite{Urakawa Hajime Spectral geometry of the Laplacian2017}
 For the Laplacian $\Delta$ of an $n$-dimensional Riemannian manifold $M^n$, %$(M,g)$ without boundary $\partial M=\varnothing$, 
if there exist $C^{\infty}$ functions $h$ and $f\neq {\rm Const}$ on $M^n$  satisfying that 
 $$
 (\Delta+h)f=0,
 $$
 then, the nodal set of $f$, $f^{-1}(0)$, is an $(n-1)$-dimensional $C^{\infty}$ submanifold of $M^n$ except an $(n-1)$-dimensional measure $0$ closed subset.
 \end{prop}
\section{The volume gap  for  closed minimal  submanifolds}\label{sect3}
Let $f: M^n\looparrowright    \mathbb{S}^{N} \subset \mathbb{R}^{N+1}$ be a closed minimal immersed submanifold  in the unit sphere  $\mathbb{S}^{N}$ and let $\mathbb{B}^{n}$  denote the unit ball  in  $\mathbb{R}^{n}$.
 \begin{defn}\label{defintion xi constant manifold}
We define a function $\xi$ on $\mathbb{S}^{N}$ and a constant $\Xi$ on $M^n$:
 $$
\xi (a)=
\liminf\limits_{t\to {1^-}}
\frac{{\rm Vol}\left\lbrace \varphi_a \geq t\right\rbrace}
{\left(1-t^2\right)^{\frac{n}{2}}{\rm Vol}(\mathbb{B}^{n})}, \ \
\Xi= \sup_{a\in f(M)}\xi (a).
$$
 \end{defn}
By Proposition \ref{prop The monotonicity formula for minimal submanifolds}, it is easy to show that $\xi (a)$ and $\Xi $ are well-defined. The following lemma gives a lower bound estimate for $\xi(a)$ and $\Xi$.

\begin{lem}\label{lemma  gap  results of k for immersion}
 For any point $p\in f(M)$,
 if its preimage set  consists of $m(p)$ distinct points in $M^n$, then  $\xi(p)\geq m(p)$. In particular, $\Xi\geq 2$ if $f$ is not an embedding.
  \end{lem}
   \begin{proof}
   Since $\varphi_p(x)=1$ if and only if $f(x)=p$, it follows that
   $$\left\lbrace x\in M^n: \varphi_p(x) =1\right\rbrace=f^{-1}(p) \not=\varnothing.$$ The properties of compactness and  local embedding of $M$ show that    $f^{-1}(p)$ is a finite set, say,
   $f^{-1}(p)=\left\lbrace q_1, q_2,..., q_{m{(p)}}\right\rbrace $ and $q_i\in M$. %By the inverse function theorem,
   For each $i$, there is a neighborhood $U_i$ of $q_i$ such that $f$ is a diffeomorphism from  $U_i$ to $f(U_i)$, and by shrinking the $U_i$'s if necessary, we may assume that they are pairwise disjoint. For any sufficiently small $\delta>0$ and for any $1-\delta<t<1$, we have
   $$\left\lbrace  \varphi_p \geq t\right\rbrace\subset \bigsqcup_{i=1}^{m{(p)}}U_i.$$
   Note that the intrinsic distance between two points of the submanifold is always greater  than or equal to the  distance of the ambient manifold. It follows that each $\left\lbrace  \varphi_p \geq t\right\rbrace\cap U_i$ contains a geodesic ball $B_r(q_i)\subset M$ of radius $r$, where $r=\arccos t$ is the radius of the geodesic ball $\{y\in\mathbb{S}^N : \langle y, p\rangle\geq t\}\subset \mathbb{S}^N$. Hence
   $$\bigsqcup_{i=1}^{m{(p)}}B_r(q_i)\subset \left\lbrace  \varphi_p \geq t\right\rbrace.$$
 Recalling the expansion formula of volume  for the geodesic ball $B_r(q_i)$ (cf. \cite{Alfred Gray 1973})
 $$
{\rm Vol} (B_r(q_i))={\rm Vol}(\mathbb{B}^{n})r^n\left[1-\frac{R_M(q_i)}{6(n+2)}r^2+O(r^4) \right],
 $$
where  $R_M$ is the scalar curvature of $M$,
 we have
   $$
     \begin{aligned}
     \xi (p)
     &=
 \liminf\limits_{t\to {1^-}}
   \frac{{\rm Vol}\left\lbrace \varphi_p \geq t\right\rbrace}
   {\left(1-t^2\right)^{\frac{n}{2}}{\rm Vol}(\mathbb{B}^{n})}
\geq
 \liminf\limits_{r\to {0^+}}
      \sum_{i=1}^{m(p)}\frac{
      {\rm Vol}(B_r(q_i))}
      {\sin ^n r{\rm Vol}(\mathbb{B}^{n})}\\
&
=\lim\limits_{r\to {0^+}}
\sum_{i=1}^{m(p)}
      \frac{r^n\left( 1+O(r^2)\right) }
      {\sin ^n r}={m{(p)}}.
     \end{aligned}
   $$
If  $f$ is not an  embedding, we have
$$
\Xi=\sup_{p\in f(M)}
\xi(p)
   \geq\sup_{p\in f(M)}{m{(p)}}\geq2.
$$
  \end{proof}

\begin{lem}\label{Lemma  gap  results of volume sr}
For any fixed unit vector $a \in \mathbb{S}^{N}$, if $M\not\subset \left\lbrace\varphi_a=0\right\rbrace$, then
$$
{\rm Vol}\left\lbrace s\leq\varphi_a \leq r\right\rbrace \geq
n\xi(a){\rm Vol}(\mathbb{B}^{n})
\int_{s}^{r}\left(1-t^2\right)^{\frac{n-2}{2}}dt,
$$
for all $0 \leq s\leq r\leq 1$.
 \end{lem}
 \begin{proof}
 By Proposition \ref{prop The monotonicity formula for minimal submanifolds},
 for $0\leq s_1\leq s_2< 1$, we have
 $$
 \frac{ \int_{\left\lbrace \varphi_a\geq s_1\right\rbrace }
 \varphi_a}{\left(1-s_1^2\right)^{\frac{n}{2}} }
 \geq
 \frac{ \int_{\left\lbrace \varphi_a\geq s_2\right\rbrace }
 \varphi_a}{\left(1-s_2^2\right)^{\frac{n}{2}} }.
 $$
 Let $s_1=t\geq0$ and $s_2\to 1^-$,
then we have
 \begin{equation}\label{eq-t0}
    \begin{aligned}
  \int_{\left\lbrace \varphi_a\geq t\right\rbrace}\varphi_a
  &\geq
 \left(1-t^2\right)^{\frac{n}{2}}
  \liminf\limits_{s_2\to {1^-}}
   \frac{ \int_{\left\lbrace \varphi_a\geq s_2\right\rbrace }
   \varphi_a}
   {\left(1-s_2^2\right)^{\frac{n}{2}} }
   \geq
 \left(1-t^2\right)^{\frac{n}{2}}
  \liminf\limits_{s_2\to {1^-}}
   \frac{ \int_{\left\lbrace \varphi_a\geq s_2\right\rbrace }
s_2}
   {\left(1-s_2^2\right)^{\frac{n}{2}} } \\
  &=
 \left(1-t^2\right)^{\frac{n}{2}}
  \liminf\limits_{s_2\to {1^-}}
  \frac{{\rm Vol}\left\lbrace \varphi_a \geq s_2\right\rbrace}
  {\left(1-s_2^2\right)^{\frac{n}{2}}}
      \liminf\limits_{s_2\to {1^-}}s_2\\
    &=\xi(a){\rm Vol}(\mathbb{B}^{n})
\left(1-t^2\right)^{\frac{n}{2}}.
  \end{aligned}
  \end{equation}
  Thus,  similar to (\ref{equation 1 leqr to geqr varphia}), we derive
  \begin{equation*}\label{equation 2 leqr to geqr varphia}
  \begin{aligned}
  \int_{\left\lbrace s\leq\varphi_a \leq r\right\rbrace }
1
  &=\int_s^r dt
  \int_{\left\lbrace \varphi_a=t\right\rbrace }
  \frac{1}{|a^{\rm T}|}
  \geq
  \int_s^r dt
  \int_{\left\lbrace \varphi_a=t\right\rbrace }
  \frac{1}{\sqrt{1-\varphi_a^2}}\\
  &=
  \int_s^r dt
  \int_{\left\lbrace \varphi_a=t\right\rbrace }
  \frac{1}{\sqrt{1-t^2}}
  \geq
  \int_s^r dt
  \int_{\left\lbrace \varphi_a\geq t\right\rbrace }
  \frac{n}{{1-t^2}}
  \varphi_a\\
%  &\geq
%  n\lim\limits_{u\to {1^-}}
%    \frac{{\rm Vol}\left\lbrace \varphi_a \geq u\right\rbrace}
%    {\left(1-u^2\right)^{\frac{n}{2}}}
%  \int_{s}^{r}\left(1-t^2\right)^{\frac{n-2}{2}}\\
&\geq
n\xi(a){\rm Vol}(\mathbb{B}^{n})
\int_{s}^{r}\left(1-t^2\right)^{\frac{n-2}{2}} dt.
  \end{aligned}
  \end{equation*}
%Finally, we can take limits on both ends of the above inequality.
 \end{proof}

% \begin{cor}\label{theorem  gap  results of volume for antipodal map}
%If  $M^n$ is invariant under the antipodal map %and there exists a point such that its preimage %set  consists of $m$ distinct points in $M^n$, %then
% $${\rm Vol}(M^n)\geq m{\rm %Vol}(\mathbb{S}^{n}).$$
%  \end{cor}
%  \begin{proof} \end{proof}

%\subsection{The volume gap  for  closed minimal  submanifolds}
%Let $f: M^n\looparrowright    \mathbb{S}^{N} \subset \mathbb{R}^{N+1}$ be a closed minimal immersed submanifold  in the unit sphere  $\mathbb{S}^{N}$ and let $\mathbb{B}^{n}$  denote the unit ball  in  $\mathbb{R}^{n}$.
 The following inequality shown in (\ref{eq-t0}) will be used later.

\begin{lem}\label{Lemma  gap  results of varphi a volume sr}%\cite{Ge Li 2022 Solomon-Yau conj}
For any fixed unit vector $a \in \mathbb{S}^{N}$, if $M\not\subset \left\lbrace\varphi_a=0\right\rbrace$, then
$$
   \int_{\left\lbrace \varphi_a\geq t\right\rbrace}\varphi_a
   \geq \xi(a){\rm Vol}(\mathbb{B}^{n})
    \left(1-t^2\right)^{\frac{n}{2}},
$$
for all $0 \leq t \leq 1$.
 \end{lem}

\begin{thm}\label{thm  f_1(s)=1+s^2, f_2(s)=s, V(s)=1-s^2}
  Let $M^n$ be a  closed minimal immersed  submanifold  in  $\mathbb{S}^{N}$. If $M\not\subset \left\lbrace\varphi_a=0\right\rbrace$, then
 $$
 	\begin{aligned}
 \int_{\{ \varphi_a \geq 0\}}(1+\varphi_{ a }^2)&\geq  
    \frac{(n+2){\rm Vol}(\mathbb{S}^n)}{2(n+1){\rm Vol}(\mathbb{B}^n)} \int_{\{ \varphi_a\geq 0\}}\varphi_a,
 \end{aligned}
 $$
   where the equality  holds if and only if $M^n$ is totally geodesic.
 \end{thm}

 \begin{proof}
 Let $f_1(s)=1+s^2, f_2(s)=s, V(s)=1-s^2$ in Lemma \ref{lem  f1f2V}, we have
 $$
 \left\{\begin{array} { l } 
 	{ \theta _ { 1 } = \frac { A _ { 1 } } { B _ { 1 } } }\\
 	{ A _ { 1 } = \int _ { \{ \varphi _ { a }\geq t \} } \left( 1+\varphi_{ a }^2\right)  } \\
 	{ B _ { 1 } = w _ { n - 1 } \int _ { t } ^ { 1 }(1+s^2) ( 1 - s ^ { 2 } ) ^ { \frac { n } { 2 } - 1 } d s } 
 \end{array} \quad \left\{\begin{array}{l}
 	\theta_2=\frac{A_2}{B_2}\\
 	A_2=\int_{\{\varphi_a\geq t\} } \varphi_a \\
 	B_2=w_{n-1} \int_t^1 s\left(1-s^2\right)^{\frac{n}{2}-1} d s%=\frac{1}{n}\left(1-t^2\right)^{\frac{n}{2}} w_{k-1} 
 \end{array}\right.\right..
 $$
 (i) By Proposition \ref{prop The monotonicity formula for minimal submanifolds}, we have
 $$
 \frac{d \theta_2}{d t}=\frac{d}{d t}\left(\frac{A_2}{B_2}\right)=\frac{d}{d t}\left(\frac{\int_{\{\varphi_a\geq t\} } \varphi_a}{\frac{\omega_{n-1}}{n}\left(1-t^2\right)^{\frac{n}{2}}}\right)\leq 0 \quad (0<t<1).
 $$
 (ii) For $0<t<1$, one has
\begin{eqnarray} 
&~&  f_1\left(\frac{B_2}{f_2}-\frac{B_1}{f_1}\right)
\nonumber \\
& =&\frac{f_1}{f_2} B_2-B_1=\frac{1+t^2}{t} B_2-B_1 \nonumber\\
& =&\frac{\omega_{n-1}(1+t^2)}{t} \int_t^1 s\left(1-s^2\right)^{\frac{n}{2}-1} d s-w_{n-1} \int_t^1(1+s^2)\left(1-s^2\right)^{\frac{n}{2}-1} d s \nonumber\\
& =&w_{n-1} \int_t^1 \frac{\left( s-t\right) \left( 1-ts\right) }{t}\left(1-s^2\right)^{\frac{n}{2}-1} d s \nonumber\\
&>& 0.\nonumber% \quad(1\geqslant  s\geq t)
\end{eqnarray}
 Due to (i) and (ii), we have
 %Hence
 $$
 \frac{d}{d t}\left(B_1\left(\theta_1-\theta_2\right)\right)=f_1 \frac{d \theta_2}{d t} \cdot\left(\frac{B_2}{f_2}-\frac{B_1}{f_1}\right) \leq 0
 \quad (0<t<1). 
 $$
 Hence, 
 $B_1\left(\theta_1-\theta_2\right)$ is monotone decreasing for $0<t<1$  and
 %$$
 %	\left.B_1\left(\theta_1-\theta_2\right)\right|_{t=0}\geq\left. B_1\left(\theta_1-\theta_2\right)\right|_{t=1^{-}} \quad\left(\lim _{t\rightarrow1^{-}}\right) 
 %$$
 $$
 	\begin{aligned}
 \left.B_1\left(\theta_1-\theta_2\right)\right|_{t=0}
 & \geq
 	\lim _{t \rightarrow 1^{-}} B_1\left(\theta_1-\theta_2\right)
 %	=\lim _{t \rightarrow 1^{-}} B_1\left(\frac{A_1}{B_1}-\frac{A_2}{B_2}\right)
 	=\lim _{t \rightarrow 1^{-}}\left(A_1-\frac{B_1}{B_2} A_2\right) \\
 	&=\lim _{t \rightarrow 1^{-}}\left(\int_{\{ \varphi_a \geq t\}} \left( 1+\varphi_{ a }^2\right) -\frac{\int_t^1(1+s^2)\left(1-s^2\right)^{\frac{n}{2}-1} d s}{\int_t^1 s\left(1-s^2\right)^{\frac{n}{2}-1} d s} \int_{\{ \varphi_a\geq t\}} \varphi_a\right) \\
 	&=0.
 \end{aligned}
 $$
 Thus
 $$
 \left.B_1\left(\theta_1-\theta_2\right)\right|_{t=0}=\int_{\{ \varphi_a \geq 0\}} 
 \left( 1+\varphi_{ a }^2\right)
 -\frac{\int_0^1(1+s^2)\left(1-s^2\right)^{\frac{n}{2}-1} d s}{\int_0^1 s\left(1-s^2\right)^{\frac{n}{2}-1} d s} \int_{\{ \varphi_a\geq 0\}} \varphi_a\geq0,
 $$
 and
 $$
 	\begin{aligned}
 \int_{\{ \varphi_a \geq 0\}}(1+\varphi_{ a }^2)&\geq \frac{\int_0^1(1+s^2)\left(1-s^2\right)^{\frac{n}{2}-1} d s}{\int_0^1 s\left(1-s^2\right)^{\frac{n}{2}-1} d s} \int_{\{ \varphi_a\geq 0\}} \varphi_a
 %=\frac{\frac{1}{2k}\cdot\frac{vol(S^k)}{vol(B^k)}}{\frac{1}{k}}\int_{\{ \varphi_a\geq 0\}} \varphi_a\\ &
 =   \frac{(n+2){\rm Vol}(\mathbb{S}^n)}{2(n+1){\rm Vol}(\mathbb{B}^n)} \int_{\{ \varphi_a\geq 0\}}\varphi_a.
 \end{aligned}
 $$
If the equality  above  holds, then   $B_1\left(\theta_1-\theta_2\right)$ is constant by its monotonicity. Hence
$$ \frac{\int_{\{\varphi_a\geq t\} } \varphi_a}{ {\rm Vol}(\mathbb{B}^n)\left(1-t^2\right)^{\frac{n}{2}}}\equiv \xi(a) \quad (0<t<1)$$ 
 by (i) and (ii). 
 If $\xi(a) =0$, then 
 $$
  	\begin{aligned}
  0<{\rm Vol}\{ \varphi_a \geq 0\}
  \leq
  \int_{\{ \varphi_a \geq 0\}}(1+\varphi_{ a }^2)
  =   \frac{(n+2){\rm Vol}(\mathbb{S}^n)}{2(n+1){\rm Vol}(\mathbb{B}^n)} \int_{\{ \varphi_a\geq 0\}}\varphi_a=0.
  \end{aligned}
  $$
  Therefore $\xi(a) >0$ and $a\in f(M)$.
 Due to the proof of  Proposition \ref{prop The monotonicity formula for minimal submanifolds}, one has
\begin{equation}\label{equation a2 psi2 vaephi2=} 
 |a^{\rm T}|^2+\varphi_a^2= 1 
\end{equation} 
on ${\{ \varphi_a \geq 0\}}$. 
Thus for all $0<t<1$,  
$ {\left\lbrace \varphi_a\geq  t\right\rbrace }$
 is a free boundary minimal submanifold of  $\mathbb{S}^N_t$, where
 $\mathbb{S}^N_t:={\left \lbrace x\in \mathbb{S}^N: \varphi_a(x)\geq t\right\rbrace }
$, i.e.,  $ {\left\lbrace \varphi_a= t\right\rbrace }$ intersects $\partial \mathbb{S}^N_t$ orthogonally. Hence,  $ {\left\lbrace \varphi_a\geq  0\right\rbrace }$ is a star-shaped minimal cone %with density $\xi(a)$ 
at the center $a$. 
For any   sufficiently small $\delta>0$ and  for any  $1-\delta<r<1$, we  have  $m=m(a)$ (multiplicity at $a\in f(M)$) disjoint connected regions $M_i\subset M$ such that
   $$\left\lbrace  \varphi_a \geq r\right\rbrace
   =
    \bigsqcup_{i=1}^{m}M_i.$$
Note that $M_i$ is a geodesic ball of $M^n$ with radius $s=\arccos r$ centered at $a\in f(M)$, since  $ {\left\lbrace \varphi_a\geq  0\right\rbrace }$ is a star-shaped minimal cone. 
    Similar to the proof of Lemma \ref{lemma  gap  results of k for immersion}, we have \begin{equation}\label{equation xi=ma} %\xi(a)=m(a).
     \begin{aligned}
     \xi (a)
     &=
 \liminf\limits_{r\to {1^-}}
   \frac{{\rm Vol}\left\lbrace \varphi_a \geq r\right\rbrace}
   {\left(1-r^2\right)^{\frac{n}{2}}{\rm Vol}(\mathbb{B}^{n})}
=
 \liminf\limits_{s\to {0^+}}
      \sum_{i=1}^{m}\frac{
      {\rm Vol}(M_i)}
      {\sin ^n s{\rm Vol}(\mathbb{B}^{n})}\\
&
=\lim\limits_{s\to {0^+}}
\sum_{i=1}^{m}
      \frac{s^n\left( 1+O(s^2)\right) }
      {\sin ^n s}={m{}}.
     \end{aligned}
\end{equation}
Notice that (\ref{equation xi=ma}) implies that the density of $ {\left\lbrace \varphi_a\geq  0\right\rbrace }$ at $a\in f(M)$ is $\xi(a)=m$. 
%intersects ?? orthogonally. and Lemma \ref{lemma  gap  results of k for immersion}  \xi(a)\geq m(a), \quad , where $m(a)$ is the multiplicity of $a\in f(M)$
%Thus $a^{\rm T}=a^{\rm T_{\mathbb{S}^N}}$
By (\ref{equation a2 psi2 vaephi2=}), (\ref{equation xi=ma}), Lemma \ref{lemma  gap  results of k for immersion} and a similar  proof of Proposition \ref{prop The monotonicity formula for minimal submanifolds}, for all $1\leq i \leq m$, we have
$$
\int_{M_i}
\varphi_a=
\int_{\partial M_i}
\frac{|a^{\rm T}|}{n}=
\int_{\partial M_i}
\frac{\sqrt{1-\varphi_a^2}}{n}=
\frac{\sqrt{1-r^2}}{n}{\rm Vol}{(\partial M_i)}$$
and
$$ \frac{\int_{M_i } \varphi_a}{{\rm Vol}(\mathbb{B}^n)\left(1-r^2\right)^{\frac{n}{2}}}\equiv  1.$$
%where  $\sum_{i=1}^{m}\xi_i(a)=\xi(a)$.
%Due to (\ref{equation xi=ma}), one has $\xi_i(a)=1$ for all $1\leq i \leq m$.
Hence
\begin{equation*}
n^n{\rm Vol}(\mathbb{B}^{n}) 
 \left(
\int_{M_i} \varphi _ { p }\right) ^{n-1}=
{\rm Vol}^n(\partial M_i)
\end{equation*}
for all  $1\leq i \leq m$. By Theorem \ref{thm Choe and Gulliver 1992 M and partial M}, one has
$\left\lbrace  \varphi_a \geq r\right\rbrace
   =
    \bigsqcup_{i=1}^{m}M_i$  is totally geodesic in $\mathbb{S}^N$. 
      As $M^n$ is minimal, $\Delta \varphi_b=-n\varphi_b$ for all $b\in \mathbb{S}^N$ by Proposition \ref{prop funda}. %$M^n$ is non-totally-geodesic and
       Since  the dimension of the first eigenfunctions $\left\lbrace \varphi_b\right\rbrace $  for  an $n$-dimensional  totally geodesic submanifold in $\mathbb{S}^N$ is $n+1$ and $n<N$, once $M^n$ is not totally geodesic globally (though the $m$ pieces $M_i$ are), we can  find a non-zero eigenfunction $\varphi_b$  for some $b\in \mathbb{S}^N$ and $M_i$ such that $b\bot M_i$, i.e.,%$$\varphi_b\equiv 0\quad  \text{on}
%      \left\lbrace  \varphi_a \geq r\right\rbrace.$$
    $$
         \varphi_b=
          0 \quad \text{on}\ 
    M_i.
%     \varphi_b=
%     \left\{
%     \begin{aligned}
%      0 &, \quad \text{on}\ 
%M_i; \\
%      \text{non-constant} &, \quad \text{on}\ 
%                  M\backslash M_i.
%      \end{aligned}
%\right.
$$
This makes a contradiction with Proposition \ref{prop smoothness of nodal set}, since 
 the nodal set of $\varphi_b$ (i.e., $\varphi_b^{-1}(0)$) includes a subset $M_i$ of dimension $n$. Hence, $M^n$ is totally geodesic in $\mathbb{S}^N$ if the equality of the theorem holds.
 \end{proof}

%By (\ref{equation a2 psi2 vaephi2}), (\ref{equation  n varphia2 leq at2}) and (\ref{equation  varphia2 leq 0}), we have 
 \begin{cor}\label{thm  f_1(s)=1, f_2(s)=s, V(s)=1-s^2}
  Let $M^n$ be a  closed minimal immersed  submanifold  in  $\mathbb{S}^{N}$. If $M\not\subset \left\lbrace\varphi_a=0\right\rbrace$, then
  $$
  \int_{\{ \varphi_a \geq 0\}}1 \geq   \frac{{\rm Vol}(\mathbb{S}^n)}{2{\rm Vol}(\mathbb{B}^n)}\int_{\{ \varphi_a\geq 0\}}\varphi_a, \quad  {\rm Vol}(M^n) \geq \frac{{\rm Vol}(\mathbb{S}^n)}{2{\rm Vol}(\mathbb{B}^n)}\int_{M}|\varphi_a|,
  $$
    where each equality  holds if and only if $M^n$ is totally geodesic.
  \end{cor}
  \begin{proof}
   By Proposition \ref{prop funda}, $$\nabla \varphi_a=a^{\rm T},\ \ \Delta \varphi_a=-n\varphi_a.$$
   Due to  the divergence theorem and
   $
   |a^{\rm T}|^2+\varphi_a^2\leq 1,
  $
  we have
   \begin{equation}\label{equation volM geq n+1varphi}
   \int_{\{ \varphi_a \geq 0\}}1 \geq (n+1) \int_{\{ \varphi_a \geq 0\}} \varphi_{a}^2, \quad {\rm Vol}(M^n) \geq (n+1) \int_{M} \varphi_{a}^2,
   \end{equation}
   and  the required inequalities follow from  Theorem \ref{thm  f_1(s)=1+s^2, f_2(s)=s, V(s)=1-s^2}.
   Since each equality of  (\ref{equation volM geq n+1varphi})  holds if and only if $ |a^{\rm T}|^2+\varphi_a^2=1$ by (\ref{equation a2 psi2 vaephi2}), % and (\ref{equation  n varphia2 leq at2}),  
     which characterizes that $M^n$ is totally geodesic by a similar proof of Theorem \ref{thm  f_1(s)=1+s^2, f_2(s)=s, V(s)=1-s^2}.
%   This completes the proof by (\ref{equation volM geq n+1varphi}) and Theorem \ref{thm  f_1(s)=1+s^2, f_2(s)=s, V(s)=1-s^2}.
  \end{proof}

%Similar to (\ref{equation  varphia2 leq 0}), we have 
%   \begin{equation}\label{equation volM geq n+1varphi}    {\rm Vol}(M^n)     \geq  (n+1)    \int_{M} \varphi_{a}^2.   \end{equation} Since 
%$    {\rm Vol}(M^n) 
% =
 % (n+1)
  %  \int_{M} \varphi_{a}^2$ if and only if $ |a^{\rm T}|^2+\varphi_a^2=1$ by (\ref{equation a2 psi2 vaephi2}) and (\ref{equation  n varphia2 leq at2}),    the equality of  (\ref{equation volM geq n+1varphi})  holds  if and only if $M^n$ is totally geodesic by a similar proof of Theorem \ref{thm  f_1(s)=1+s^2, f_2(s)=s, V(s)=1-s^2}.
 \begin{proof}[\textbf{Proof of Theorem \ref{theorem  gap  results of volume in the m preimage varphi p2}}]
 Without loss of generality, suppose
 $M\not\subset \left\lbrace\varphi_a=0\right\rbrace$ for any $a\in \mathbb{S}^N$.
% Setting $s=0$ and $r=1$ in Lemma \ref{Lemma  gap  results of volume sr}, we derive from Lemma \ref{lemma  gap  results of k for immersion} that for any $a\in f(M)$, 
 %  \begin{equation*} 
%   \int_{\left\lbrace \varphi_a\geq 0\right\rbrace  }
%  1
%   \geq
%   n\xi(a){\rm Vol}(\mathbb{B}^{n})
%   \int_{0}^{1}\left(1-t^2\right)^{\frac{n-2}{2}}dt\geq
%   \frac{m(a)}{2}{\rm Vol}(\mathbb{S}^{n}).
%   \end{equation*}
    By Lemma \ref{Lemma  gap  results of varphi a volume sr},
    we have
 $$
 \int_{\left\lbrace \varphi_a\geq 0\right\rbrace}\varphi_a \geq \xi(a){\rm Vol}(\mathbb{B}^{n})\geq 
    m(a){\rm Vol}(\mathbb{B}^{n}).
    $$
   By Proposition \ref{prop funda}, for any $a\in {\mathbb{S}^{N}}$, $\int_{M}\varphi_{a}=0,$ thus
   \begin{equation}\label{equation varphi s geq ma Bn}
   \int_{{\left\lbrace \varphi_a\geq 0\right\rbrace  }}\varphi_{a}=\int_{\left\lbrace \varphi_a
   \leq 0\right\rbrace  }
   -\varphi_a\geq 
      m(a){\rm Vol}(\mathbb{B}^{n}).%=\int_{{\left\lbrace \varphi_{-a}\geq 0\right\rbrace  }}\varphi_{-a}.
   \end{equation}
  If there exists a point $p\in f(M)$ such that its preimage set  consists of $m$ distinct points in $M^n$, then
 combining (\ref{equation volM geq n+1varphi}),  (\ref{equation varphi s geq ma Bn}) and Theorem \ref{thm  f_1(s)=1+s^2, f_2(s)=s, V(s)=1-s^2}, we obtain
 $$
   {\rm Vol}(M^n) 
     \geq  
  \frac{n+1}{n+2}
 \int_{M}\left( 1+\varphi_{p}^2\right) 
 \geq 
 m{\rm Vol}(\mathbb{S}^{n}),$$
where each equality holds if and only if $M^n$ is totally geodesic.
  \end{proof}
 % {theorem  gap  results of volume nonembedded IE n dim to n+1 sphere}
 \begin{proof}[\textbf{Proof of Corollary {\ref{theorem  gap  results of volume nonembedded hypersurface}}}] 
By Proposition \ref{prop funda} and    $
   |a^{\rm T}|^2+\varphi_a^2+\psi_{a}^2= 1
  $, we have
  $$
    (n+1)
      \int_{M} \varphi_{a}^2+\int_{M} \psi_{a}^2
      =      {\rm Vol}(M^n) 
  $$
  for all $a \in \mathbb{S}^{n+1}$. 
  This completes the proof
 by Theorem \ref{theorem  gap  results of volume in the m preimage varphi p2}.
   \end{proof}
   \begin{proof}[\textbf{Proof of Corollary {\ref{theorem  gap  results of volume nonembedded IE n dim to n+1 sphere}}}]
   By  Proposition \ref{prop funda}, Proposition \ref{prop funda psi} and Theorem \ref{thm-IE-hypers-charact}, we have
\begin{equation*}
S_{\max}\int_M \psi_a^2\geq \int_M S\psi_a^2=\int_M |Aa^{\rm T}|^2 
=\int_{M}\frac{S}{n}|a^{\rm T}|^2.
\end{equation*}
Due to
$
\int_{M }
{n}\varphi_a^2=
\int_{M}{|a^{\rm T}|^2}$ and $ \varphi_a^2+ |a^{\rm T}|^2+\psi_{a}^2=1
$,
 one has
 \begin{equation*}
 S_{\max}\int_M \psi_a^2\geq 
\frac{{S_{\min}} }{n+1} \int_{M}\left( |a^{\rm T}|^2+{\varphi_{a}}^2\right) 
=
\frac{{S_{\min}} }{n+1} \int_{M}\left( 1-{\psi_{a}}^2\right) .
 \end{equation*}
 Hence
 \begin{equation*}
\int_M \psi_a^2\geq \frac{S_{\min}}{ (n+1)S_{\max}+ S_{\min}}{{\rm Vol}(M^n)}.
 \end{equation*}
 This completes the proof by Corollary \ref{theorem  gap  results of volume nonembedded hypersurface}.
  \end{proof}

\section{The volume gap  for  closed minimal  hypersurfaces}\label{sect4}
In this section, we give
 volume gaps for both immersed and embedded closed minimal  hypersurfaces  in $\mathbb{S}^{n+1}$ under some conditions.  Firstly, we prove  Theorem \ref{theorem  gap  results of volume for hypersurface  constant S f3} and we need the following lemmas.
\begin{lem}\label{lemma GENERALIZED HADAMARD Thm  for immersion}\cite{Frankel 1966}
 Let $N^{n+1}$  be a complete, connected manifold with positive Ricci curvature. Let $V^{n}$ and $W^{n}$ be immersed minimal hypersurfaces of $N^{n+1}$, each immersed as a closed subset, and let $V^{n}$ be compact. Then $V^{n}$ and $W^{n}$ must intersect.
  \end{lem}
\begin{lem}\label{lemma square of varphi}
Let $f: M^n\rightarrow \mathbb{S}^{n+1}$ be a closed minimal  immersed  hypersurface. There exists a unit vector $a \in f(M)$ such that
 \begin{equation}\label{equation square of varphia more that volSn}
\int_M\varphi_a^2
\geq
\frac{1}{n+1}{\rm Vol}(\mathbb{S}^{n}).
\end{equation}
 In particular,  if $M^n$ is invariant under the antipodal map, then $(\ref{equation square of varphia more that volSn})$ holds for any unit vector $a \in f(M)$.
 \end{lem}
 \begin{proof}
 By Lemma \ref{lemma GENERALIZED HADAMARD Thm  for immersion}  (or Hsiang \cite{Hsiang 1967}), there exists a unit vector $a \in f(M)$ such that $-a\in f(M)$, since otherwise, we can find  two disjoint minimal hypersurfaces by antipodal map.  By the   co-area formula, (\ref{equation Stokess formula to varphia}), Proposition \ref{prop The monotonicity formula for minimal submanifolds} and Lemma \ref{lemma  gap  results of k for immersion}, we have
 \begin{equation}\label{equation square of varphia}
   \begin{aligned}
   \int_{M}
\varphi_a^2
   &=\int_0^1 dt
   \int_{\left\lbrace |\varphi_a|=t\right\rbrace }
   \frac{\varphi_a^2}{|a^{\rm T}|}
   \geq
   \int_0^1 dt
   \int_{\left\lbrace |\varphi_a|=t\right\rbrace }
   \frac{\varphi_a^2}{\sqrt{1-\varphi_a^2}}\\
   &=
   \int_0^1 dt
   \int_{\left\lbrace |\varphi_a|=t\right\rbrace }
   \frac{t^2}{\sqrt{1-t^2}}
   \geq
   \int_0^1 dt
   \int_{\left\lbrace |\varphi_a| \geq t\right\rbrace }
   \frac{nt^2}{{1-t^2}}
   |\varphi_a|\\
   &\geq
n\lim\limits_{u\to {1^-}}
     \frac{{\rm Vol}\left\lbrace |\varphi_a| \geq u\right\rbrace}
     {\left(1-u^2\right)^{\frac{n}{2}}}
   \int_{0}^{1}t^2\left(1-t^2\right)^{\frac{n-2}{2}}dt\\
    &=
        \left(\xi(a)+\xi(-a) \right) n{\rm Vol}(\mathbb{B}^{n})\int_{0}^{1}t^2\left(1-t^2\right)^{\frac{n-2}{2}}dt\\
   &\geq
     2n{\rm Vol}(\mathbb{B}^{n})\int_{0}^{1}t^2\left(1-t^2\right)^{\frac{n-2}{2}}dt\\
 &=
 \frac{1}{n+1}{\rm Vol}(\mathbb{S}^{n}).
   \end{aligned}
   \end{equation}
 \end{proof}

 \begin{thm}\label{theorem  gap  results of volume for hypersurface  constant S f3}
  Let $M^n$ be a non-totally-geodesic closed minimal immersed  hypersurface  in  $\mathbb{S}^{n+1}$ with constant scalar curvature. If the third mean curvature is constant (or $M^n$ is Integral-Einstein, see Definition  \ref{defintion integral Eins manifold}), then
  $$
  {\rm Vol}(M^n)
 \geq
 \frac{n+2}{n+1}
  {\rm Vol}(\mathbb{S}^{n}).$$
   \end{thm}
 \begin{proof}%[\textbf{Proof of Theorem {\ref{theorem  gap  results of volume for hypersurface  constant S f3}}}]
%For any $a \in \mathbb{S}^{n+1}$,  the height functions (cf. \cite{Ge Li 2020,Ge Li 2021 Perdomo conjecture}, etc.) are defined as
%\begin{equation*}\label{height functions}
%\varphi_a(x) = \langle f(x),a \rangle,\quad  \psi_a (x)=\langle \nu,a \rangle,
%\end{equation*}
%where $\nu$ is the unit normal vector field along $x\in M^n$.
Theorem \ref{thm-IE-hypers-charact}  has shown that for a non-totally-geodesic closed minimal immersed  hypersurface $M^n$  in  $\mathbb{S}^{n+1}$ with constant scalar curvature,
 \emph{$M^n$} is IE if and only if one of the following equivalent conditions holds:
 \begin{itemize}
 \item
$
 \int_{M}\varphi_a^2=\frac{1}{n+2}{\rm Vol }(M^n)$ for all  $a \in \mathbb{S}^{n+1};
$
 \item
 $
 \int_{M}\psi_a^2=\frac{1}{n+2}{\rm Vol }(M^n)$ for all  $ a \in \mathbb{S}^{n+1};
 $
 \item
 $
 \int_{M}\varphi_a^2=\int_{M}\psi_a^2 $ for all  $a \in \mathbb{S}^{n+1};
 $
 \item
 $
 \int_{M}\varphi_a \psi_a f_3=0 $ for all  $ a \in \mathbb{S}^{n+1},
 $
  where $f_3={\rm Tr}(A^3)=3\binom{n}{3}H_3$, $A$ is the shape operator with respect to the
   unit normal vector field $\nu$  and $H_3$ is the third mean curvature.
 \end{itemize}
If further $M$ has the constant squared length of second fundamental form  $S:=|A|^2>n$ and has constant third mean curvature $H_3$, then by the fourth condition above, $M$ is IE (cf. \cite{Ge Li 2020}). This is because $\varphi_a$ and $\psi_a$ are eigenfunctions of eigenvalues $n$ and $S$ respectively, and thus they are orthogonal.
 Hence,  Lemma \ref{lemma square of varphi} and the first condition imply that for these IE hypersurfaces, there exists a unit vector $a \in f(M)$ such that
  $$
  {\rm Vol}(M^n)
  =\left( n+2\right) \int_{M}\varphi_a^2
 \geq
 \frac{n+2}{n+1}
  {\rm Vol}(\mathbb{S}^{n}).
  $$
On the other hand, Simons' inequality \cite{Simons68} shows that if $0\leq S\leq n$, then either  $S\equiv0$ or $S\equiv n$  on $M^n$.
The case of  $S\equiv n$   was characterized by Chern-do Carmo-Kobayashi \cite{Chern do Carmo Kobayashi 1970} and Lawson \cite{Lawson 1969}  independently:  the Clifford torus $M_{k,n-k}$ $\left(1 \leq k \leq n-1\right) $ are the only closed minimal hypersurfaces in $\mathbb{S}^{n+1}$  with
   $S \equiv n$. It is easy to verify that ${\rm Vol}(M_{k,n-k})\geq \frac{n+2}{n+1}  {\rm Vol}(\mathbb{S}^{n})$.% by the volume formulas in the proof of Remark \ref{remark Stirling approximation}.
 \end{proof}
 Without assuming constant scalar curvature, i.e.,  $S\not\equiv Constant$, we are able to obtain Theorem \ref{theorem  gap  results of volume for hypersurface  non-constant S antipodal map} and  Theorem \ref{theorem  gap  results of volume for hypersurface  non-constant S}.
 % which implies Theorem \ref{theorem  gap  results of volume for hypersurface constant S and antipodal map}.

 Suppose
   $$ S_{\max}=\sup_{p\in M^n}S(p), \  \  S_{\min}=\inf_{p\in M^n}S(p), \  \
   C(n,S)=\max\{\theta_1,\theta_2\},$$ where
 $$
 \theta_1=\frac{\int_{M}S}
 {2nS_{\max} {\rm Vol }\left( M^n \right) },\ \
 \theta_2=\frac{n}{4n^2-3n+1}
  \frac{
  \left( {\int_{M} }S\right) ^2}
 { {\rm Vol }\left( M^n \right) \int_{M} S^2}.
 $$

 \begin{lem} \cite{Ge Li 2020} \label{lem Volume estimation of minimal hypersurface}
 Let $ M^n$ be a closed  minimal hypersurface  in $\mathbb{S}^{n+1}$.
 \begin{itemize}
 \item[(i)]
 If $ S\not\equiv0$, then
 $$
 \frac{\int_{M}S}
 {2n S_{\max}}
 \leq
 \inf_{a\in \mathbb{S}^{n+1}}\int_{M} \varphi^2_a.
 $$
 The equality holds if and only if $ S\equiv n$ and $M$  is the minimal Clifford torus $S^{1}(\sqrt{\frac{1}{n}})\times S^{n-1}(\sqrt{\frac{n-1}{n}})$.
 \item[(ii)]
 $$
 \frac{n}{4n^2-3n+1} %{\cdot}
 %\frac{\left( {\int_{M} }S\right) ^2}
 %{\int_{M} S^2}
 \left( {\int_{M} }S\right) ^2
 \leq \int_{M} S^2
 \inf_{a\in \mathbb{S}^{n+1}}\int_{M} \varphi^2_a.
 $$
 The equality holds if and only if $M$ is an equator.
 \end{itemize}
 \end{lem}

  \begin{thm}\label{theorem  gap  results of volume for hypersurface  non-constant S antipodal map}
   Let $f: M^n\rightarrow \mathbb{S}^{n+1}$ be a non-totally-geodesic closed minimal  immersed  hypersurface. If $M$ is invariant under the antipodal map, then
   $$
   {\rm Vol}(M^n)
  \geq
  \frac{1}{1-C(n,S)}
   {\rm Vol}(\mathbb{S}^{n})\geq
  \frac{2nS_{\max}}{2nS_{\max}-S_{\min}}
   {\rm Vol}(\mathbb{S}^{n}).
   $$
   In particular, if $S$ is constant, then
   $C(n,S)=\frac{1}{2n}$.
    \end{thm}
 \begin{proof}%[\textbf{Proof of Theorem $\mathbf{\ref {theorem  gap  results of volume for hypersurface  non-constant S}}$}]
 By Proposition \ref{prop funda} and $\Delta \varphi_a^2=-2n\varphi_a^2+2|a^{\rm T}|^2$, one has
 \begin{equation}\label{equation nvarphia2 aT}
n\int_
{M }
\varphi_a^2=
\int_{M }
|a^{\rm T}|^2.
\end{equation}
Note that $|a^{\rm T}|^2+\varphi_a^2+\psi_a^2=1$ and (\ref{equation square of varphia more that volSn}) holds for  all $a \in f(M)$ when $f(M)$ is symmetric about the origin. Then by %the proof of Lemma \ref{lemma square of varphi}
(\ref{equation square of varphia more that volSn})  and (\ref{equation nvarphia2 aT}), we have
 \begin{equation}\label{equation psi S M}
\int_M\left( 1-\psi_a^2\right)
\geq
{\rm Vol}(\mathbb{S}^{n}),
\end{equation}
for  all $a \in f(M)$. Integrating $a=f(y)$ over $y\in M$ on both sides of (\ref{equation psi S M}), we have
 \begin{equation}\label{equation varphiMS2}
{\rm Vol}(M^n)-{\rm Vol}(\mathbb{S}^{n})
\geq
\frac{
\int_{x\in M}\int_{y\in M}
\langle \nu_x, f(y)\rangle^2}{{\rm Vol}(M^n)}
\geq
 \inf_{a\in \mathbb{S}^{n+1}}\int_{M} \varphi^2_a.
\end{equation}
By Lemma \ref{lem Volume estimation of minimal hypersurface}, we have
\begin{equation}\label{equation varphiMS2Lower}
 \inf_{a\in \mathbb{S}^{n+1}}\int_{M} \varphi^2_a \geq C(n,S){\rm Vol}(M^n)
 \geq
 \frac{\int_{M}S}
  {2n S_{\max}}\geq
  \frac{S_{\min}}
    {2n S_{\max}}{\rm Vol}(M^n).
\end{equation}
Combining (\ref{equation varphiMS2}) and (\ref{equation varphiMS2Lower}) completes the proof.
\end{proof}
As applications, we obtain the following rigidity results.
\begin{cor}\label{corollary gap  results of volume for hypersurface  non-constant SUX  antipodal map}
  Let $M^n$ be a closed minimal  immersed  hypersurface  in  $\mathbb{S}^{n+1}$ which is invariant under the antipodal map. For any
  $\delta\geq 0$, if
$n\leq S\leq n+\delta,$ then
$$
{\rm Vol}(M^n)
\geq\frac{2\left(n+\delta\right)  }{2\left(n+\delta\right) -1 }  {\rm Vol}(\mathbb{S}^{n}).
$$
   \end{cor}
\begin{proof}
Due to $n\leq S\leq n+\delta$, we have
$$\frac{S_{\min}}
    {S_{\max}}\geq
    \frac{n}{n+\delta}.
    $$
By Theorem \ref{theorem  gap  results of volume for hypersurface  non-constant S}, one has
  $$
  {\rm Vol}(M^n)
\geq
  \frac{2nS_{\max}}{2nS_{\max}-S_{\min}}
  {\rm Vol}(\mathbb{S}^{n})\geq \frac{2\left(n+\delta\right)  }{2\left(n+\delta\right) -1 }  {\rm Vol}(\mathbb{S}^{n}).$$
  \end{proof}
%To prove Corollary \ref{corollary gap  results of volume for hypersurface  non-constant S}, we need the following  lemmas.
%\begin{lem} {\rm\textbf{$($Choi and Wang \label{lem  the first Dirichlet eigenvalue of the Laplacian Choi Wang 1983} \cite{Choi Wang 1983}$)$}}
%Let $ M^n$ be a closed embedded minimal hypersurface in $\mathbb{S}^{n+1}$ and $\lambda_1(M)$ be the first positive eigenvalue of the Laplacian, then
%$\lambda_1 (M)\geq n/2$.
%\end{lem}
%A careful argument (see \cite[ Theorem 5.1]{Brendle S 2013 survey of recent results}) shows that the strict inequality holds, i.e., $\lambda_1 (M)> n/2$ in Lemma \ref{lem  the first Dirichlet eigenvalue of the Laplacian Choi Wang 1983}.

%\begin{lem}\label{lem A similar Simons Inequality of minimal hypersurfaces in the spheres}\cite{Ge Li 2021 Perdomo conjecture}
%Let $ M^n$ be a closed   immersed minimal hypersurface  in  $\mathbb{S}^{n+1}$, then
%\begin{equation*}
%\begin{aligned}
%\int_{M}S\left(S-n\right)
%(S-S_{\min})
%&\geq
%\frac{\lambda_1(M)}{2}\left( \int_{M}S^2-
%\frac{\left( \int_{M}S\right) ^2}
%{{\rm Vol}(M^n)}\right).
%\end{aligned}
%\end{equation*}
%\end{lem}

  \begin{cor}\label{corollary gap  results of volume for hypersurface  non-constant S antipodal}
  Let $M^n$ be a closed minimal  immersed  hypersurface  in  $\mathbb{S}^{n+1}$ which is invariant under the antipodal map. For any $ \delta\leq\frac{3}{8}n$, if %the squared length of the second fundamental form and the volume of $M$ satisfy
 the following conditions are satisfied:
  \begin{itemize}
  \item[(i)]  $S\leq n+\delta\leq  \frac{11}{8}n,$
  \item[(ii)]  $  {\rm Vol}(M^n)
 \leq\frac{3\left( 4n^2-3n+1\right) }{3(4n^2-4n+1)+8\delta }  {\rm Vol}(\mathbb{S}^{n}),$
 \end{itemize}
 then   $M$ is totally geodesic.
   \end{cor}

\begin{proof}%[\textbf{Proof of Corollary $\mathbf{\ref{corollary gap  results of volume for hypersurface  non-constant S}}$}]
Without loss of generality, we suppose   $M^n$ is a non-totally geodesic closed minimal  embedded  hypersurface  in  $\mathbb{S}^{n+1}$ because of  Corollary \ref{thm 2 gap  results of volume in the introduction must be embedded}. Besides, if $S$ is constant, then a contradiction follows directly from Theorem \ref{theorem  gap  results of volume for hypersurface  non-constant S antipodal map}.
Let $h$
denote     the second fundamental
form of  hypersurface with respect to the unit normal vector field  $\nu$.
If $\left\lbrace \omega_1, \omega_2, \cdots, \omega_n\right\rbrace$ is a local orthonormal coframe field, then $h$ can be written as
$$
h=\sum_{i,j}h_{ij}\omega_i\otimes\omega_j.
$$
The covariant derivative $\nabla h$ with components $h_{ijk}$ is given by
$$
\sum_{k}h_{ijk}\omega_k=dh_{ij}+\sum_{k}h_{kj}\omega_{ik}+\sum_{k}h_{ik}\omega_{jk},
$$
and  $\left\lbrace \omega_{ij}\right\rbrace $
are the connection forms of $M$ with respect to  $\left\lbrace \omega_1, \omega_2, \cdots, \omega_n\right\rbrace$, which  satisfy the following structure equations:
$$
d\omega_i=-\sum_{j}\omega_{ij}\wedge\omega_j,\ \  \omega_{ij}+\omega_{ji}=0,
$$
$$
d\omega_{ij}=-\sum_{k}\omega_{ik}\wedge\omega_{kj}+\frac{1}{2}\sum_{k,l}R_{ijkl}\omega_{k}\wedge\omega_{l},
$$
where $\{R_{ijkl}\}$ are  the coefficients of the Riemannian curvature tensor on $M$.  %We have the Gauss and Codazzi euqations:
%$$
%R_{ijkl}=\delta_{ik}\delta_{jl}-\delta_{il}\delta_{jk}+h_{ik}h_{jl}-h_{il}h_{jk},
%$$
%and
%$$
%h_{ijk}=h_{ikj}.
%$$
Hence
$$
S=\sum_{i,j}h_{ij}^2, \quad
|\nabla h|^2=\sum_{i,j,k}h_{ijk}^2.
$$
By the Cauchy-Schwarz inequality, one has
\begin{equation}\label{equation Sh}
|\nabla S|^2=4\sum_{k}\left( \sum_{i,j}h_{ij}h_{ijk}\right) ^2\leq
4\sum_{k}\left( \sum_{i,j}h_{ij}^2\sum_{i,j}h_{ijk}^2\right) =4S|\nabla h|^2.
\end{equation}
%If $S$ is constant, then a contradiction follows directly from Theorem \ref{theorem  gap  results of volume for hypersurface  non-constant S}.
By Simons' identity \cite{Simons68},  we have \begin{equation}\label{equation Simons equation 1}
 \frac{1}{2}\Delta S=|\nabla h|^2+S(n-S).
\end{equation}
Due to (\ref{equation Simons equation 1}) and
$$
\frac{1}{2}\Delta S^2=S\Delta S+|\nabla S|^2,
$$
we obtain
\begin{equation}\label{equation Simons equation S2}
\int_{M} |\nabla S|^2=2\int_{M}\left( S^2(S-n)-S|\nabla h|^2\right).
\end{equation}
By (\ref{equation Sh}) and (\ref{equation Simons equation S2}), one has
$\int_{M}  S^2(S-n)\leq  3\int_{M}  S|\nabla h|^2$ and
\begin{equation}\label{equation SS}
\int_{M} |\nabla S|^2\leq \frac{4}{3}\int_{M}  S^2(S-n)
\leq \frac{4}{3}(S_{\max}-n)\int_{M}  S^2.
\end{equation}
Since  $S$ is not constant,  using Rayleigh's formula, one has
\begin{equation}\label{equation SX}
\int_{M} |\nabla S|^2\geq  {\lambda_1(M)}\left( \int_{M}S^2-
\frac{\left( \int_{M}S\right) ^2}
{{\rm Vol}(M^n)}\right),
\end{equation}
where $\lambda_1(M)$ is  the first positive eigenvalue of the Laplacian. In addition,
Choi and Wang \cite{Choi Wang 1983}  proved that  $\lambda_1(M)\geq n/2$.
A careful argument (see \cite[ Theorem 5.1]{Brendle S 2013 survey of recent results}) showed that the strict inequality holds, i.e., $\lambda_1 (M)> n/2$.
By (\ref{equation SS}) and (\ref{equation SX}), we obtain
$$
\frac{4}{3}(S_{\max}-n)\int_{M}  S^2>
 {\frac{n}{2}}\left( \int_{M}S^2-
\frac{\left( \int_{M}S\right) ^2}
{{\rm Vol}(M^n)}\right).
$$
Hence
$$
\frac{\left( \int_{M}S\right) ^2}
{ {\rm Vol}(M^n)\int_{M}S^2}>1-\frac{8}{3n}
 \left(S_{\max}-n \right) \geq
 \frac{3n-8\delta }{3n}\geq 0,
$$
and
\begin{equation}\label{equation CnS thera2 lower bound}
C(n,S)\geq
\theta_2=\frac{n}{4n^2-3n+1}
 \frac{
 \left( {\int_{M} }S\right) ^2}
{ {\rm Vol }\left( M^n \right) \int_{M} S^2}
>\frac{3n-8\delta }{3\left( 4n^2-3n+1\right) }.
\end{equation}
%The proof is completed
By Theorem \ref{theorem  gap  results of volume for hypersurface  non-constant S}, we get
  $$
  {\rm Vol}(M^n)
>\frac{3\left( 4n^2-3n+1\right) }{3(4n^2-4n+1)+8\delta  }
  {\rm Vol}(\mathbb{S}^{n}),$$
which is  a contradiction to the assumption of volume.
\end{proof}
\begin{rem}
The following rigidity result is well known (cf. \cite{Qi Ding  Y.L. Xin. 2011,Li Lei Hongwei Xu  Zhiyuan Xu 2021,Peng and Terng1 1983}, etc.):\\
  Let $M^n$ be a closed minimal  immersed  hypersurface  in  $\mathbb{S}^{n+1}$ with  $n\leq S\leq n+\delta$. If $\delta \leq  \frac{n}{18}$,
% there exists a positive constant
%   $\delta(n)$ depending only on $n$,
%   such that if $n\leq S\leq n+\delta(n)$, $n\leq 5$,
   then $S\equiv n$ and $M^n$ is a Clifford torus.

In fact, due to some counterexamples of Otsuki \cite{Otsuki 1970},
  the condition $S\geq n$ is essential in the pinching result above.
    Compared with Corollary \ref{corollary gap  results of volume for hypersurface  non-constant S antipodal},  we have a larger pinching constant $\frac{3}{8}n$ and do not need $S\geq n$, but we need to limit the symmetry and volume.
\end{rem}
%\subsection{The volume gap  for  closed minimal  hypersurfaces}
%In this section, we give  volume gaps for both immersed and embedded closed minimal  hypersurfaces  in $\mathbb{S}^{n+1}$ under some conditions.  Firstly, we prove  Theorem \ref{theorem  gap  results of volume for hypersurface  constant S f3} and we need the following lemmas.
%\subsubsection{Integral-Einstein  hypersurfaces in spheres}
% In this section, we first prove the following.

%\subsubsection{Proof of the volume gap  for  closed minimal  hypersurfaces}
% Without assuming constant scalar curvature, i.e.,  $S\not\equiv Constant$, we are able to obtain Theorem \ref{theorem  gap  results of volume for hypersurface  non-constant S} which implies Theorem \ref{theorem  gap  results of volume for hypersurface constant S and antipodal map}.

If symmetry is not assumed, we have the following results.
  \begin{thm}\label{theorem  gap  results of volume for hypersurface  non-constant S}
   Let $f: M^n\rightarrow \mathbb{S}^{n+1}$ be a non-totally-geodesic closed minimal  immersed  hypersurface. Then
   $$
   {\rm Vol}(M^n)
  \geq
  \frac{n+2}{n+2-C(n,S)}
   {\rm Vol}(\mathbb{S}^{n})\geq
  \frac{2n(n+2)S_{\max}}{2n(n+2)S_{\max}-S_{\min}}
   {\rm Vol}(\mathbb{S}^{n}).
   $$
   In particular, if $S$ is constant, then
   $C(n,S)=\frac{1}{2n}$.
    \end{thm}
 \begin{proof}%[\textbf{Proof of Theorem $\mathbf{\ref {theorem  gap  results of volume for hypersurface  non-constant S}}$}]
 The proof is similar to the proof of Theorem \ref{theorem  gap  results of volume for hypersurface  non-constant S antipodal map}.
 By Corollary \ref{theorem  gap  results of volume nonembedded hypersurface}, we have
 \begin{equation} \label{equation gap  results of volume embedded hypersurface}
   {\rm Vol}(M^n) 
   \geq 
   {\rm Vol}(\mathbb{S}^{n})+
    \frac{1}{n+2}
    \int_{M} \psi_{a}^2,
 \end{equation}    
 for all $a \in f(M)$.
 Integrating $a=f(y)$ over $y\in M$ on both sides of (\ref{equation gap  results of volume embedded hypersurface}), we have
 \begin{equation}\label{equation varphiMS2 new1}
{\rm Vol}(M^n)-{\rm Vol}(\mathbb{S}^{n})
\geq
\frac{
\int_{x\in M}\int_{y\in M}
\langle \nu_x, f(y)\rangle^2}{{(n+2)}{\rm Vol}(M^n)}
\geq
    \frac{1}{n+2}
 \inf_{a\in \mathbb{S}^{n+1}}\int_{M} \varphi^2_a.
\end{equation}
%By Lemma \ref{lem Volume estimation of minimal hypersurface}, we have
%\begin{equation}\label{equation varphiMS2Lower}
% \inf_{a\in \mathbb{S}^{n+1}}\int_{M} \varphi^2_a \geq C(n,S){\rm Vol}(M^n)  \geq  \frac{\int_{M}S}   {2n S_{\max}}\geq   \frac{S_{\min}}     {2n S_{\max}}{\rm Vol}(M^n).
%\end{equation}
This completes the proof by
combining (\ref{equation varphiMS2 new1}) and (\ref{equation varphiMS2Lower}).
\end{proof}

 As applications, we obtain the following rigidity results.
\begin{cor}\label{corollary gap  results of volume for hypersurface  non-constant SUX}
  Let $M^n$ be a closed minimal  immersed  hypersurface  in  $\mathbb{S}^{n+1}$. For any
  $\delta\geq 0$, if
$n\leq S\leq n+\delta,$ then
$$
{\rm Vol}(M^n)
\geq
\frac{2(n+2)\left(n+\delta\right)  }{2(n+2)\left(n+\delta\right) -1 }  {\rm Vol}(\mathbb{S}^{n}).
$$
   \end{cor}
\begin{proof}
Due to $n\leq S\leq n+\delta$, we have
$$\frac{S_{\min}}
    {S_{\max}}\geq
    \frac{n}{n+\delta}.
    $$
By Theorem \ref{theorem  gap  results of volume for hypersurface  non-constant S}, one has
  $$
  {\rm Vol}(M^n)
\geq
  \frac{2n(n+2)S_{\max}}{2n(n+2)S_{\max}-S_{\min}}
  {\rm Vol}(\mathbb{S}^{n})\geq \frac{2(n+2)\left(n+\delta\right)  }{2(n+2)\left(n+\delta\right) -1 }  {\rm Vol}(\mathbb{S}^{n}).$$
  \end{proof}
%To prove Corollary \ref{corollary gap  results of volume for hypersurface  non-constant S}, we need the following  lemmas.
%\begin{lem} {\rm\textbf{$($Choi and Wang \label{lem  the first Dirichlet eigenvalue of the Laplacian Choi Wang 1983} \cite{Choi Wang 1983}$)$}}
%Let $ M^n$ be a closed embedded minimal hypersurface in $\mathbb{S}^{n+1}$ and $\lambda_1(M)$ be the first positive eigenvalue of the Laplacian, then
%$\lambda_1 (M)\geq n/2$.
%\end{lem}
%A careful argument (see \cite[ Theorem 5.1]{Brendle S 2013 survey of recent results}) shows that the strict inequality holds, i.e., $\lambda_1 (M)> n/2$ in Lemma \ref{lem  the first Dirichlet eigenvalue of the Laplacian Choi Wang 1983}.

%\begin{lem}\label{lem A similar Simons Inequality of minimal hypersurfaces in the spheres}\cite{Ge Li 2021 Perdomo conjecture}
%Let $ M^n$ be a closed   immersed minimal hypersurface  in  $\mathbb{S}^{n+1}$, then
%\begin{equation*}
%\begin{aligned}
%\int_{M}S\left(S-n\right)
%(S-S_{\min})
%&\geq
%\frac{\lambda_1(M)}{2}\left( \int_{M}S^2-
%\frac{\left( \int_{M}S\right) ^2}
%{{\rm Vol}(M^n)}\right).
%\end{aligned}
%\end{equation*}
%\end{lem}

  \begin{cor}\label{corollary gap  results of volume for hypersurface  non-constant S new}
  Let $M^n$ be a closed minimal  immersed  hypersurface  in  $\mathbb{S}^{n+1}$. For any $ \delta\leq\frac{3}{8}n$, if %the squared length of the second fundamental form and the volume of $M$ satisfy
 the following conditions are satisfied:
  \begin{itemize}
  \item[(i)]  $S\leq n+\delta\leq  \frac{11}{8}n,$
  \item[(ii)]    $
    {\rm Vol}(M^n)
\leq \frac{3(n+2)\left( 4n^2-3n+1\right) }{3(n+2)\left( 4n^2-3n+1\right)+8\delta-3n  }
    {\rm Vol}(\mathbb{S}^{n}),$
 \end{itemize}
 then   $M$ is totally geodesic.
   \end{cor}

\begin{proof}%[\textbf{Proof of Corollary $\mathbf{\ref{corollary gap  results of volume for hypersurface  non-constant S}}$}]
 The proof is similar to the proof of  Corollary \ref{corollary gap  results of volume for hypersurface  non-constant S antipodal}.
Without loss of generality, we suppose   $M^n$ is a non-totally geodesic closed minimal  embedded  hypersurface  in  $\mathbb{S}^{n+1}$ because of  Corollary \ref{thm 2 gap  results of volume in the introduction must be embedded}. Besides, if $S$ is constant, then a contradiction follows directly from Theorem \ref{theorem  gap  results of volume for hypersurface  non-constant S}.
If $S$ is not constant,
%The proof is completed
 we have
  $$
  {\rm Vol}(M^n)
>\frac{3(n+2)\left( 4n^2-3n+1\right) }{3(n+2)\left( 4n^2-3n+1\right)+8\delta-3n  }
  {\rm Vol}(\mathbb{S}^{n})$$
  by Theorem \ref{theorem  gap  results of volume for hypersurface  non-constant S} and (\ref{equation CnS thera2 lower bound}), which is
   a contradiction to the assumption of volume.
\end{proof}
\begin{rem}
    Compared with Corollary \ref{corollary gap  results of volume for hypersurface  non-constant S antipodal}, now we do not need the symmetry of hypersurfaces in Corollary \ref{corollary gap  results of volume for hypersurface  non-constant S new}.
\end{rem}

 \begin{proof}[\textbf{Proof of Theorem \ref{theorem  gap  results of volume for hypersurface constant S}}] 
Since $S$ is constant,  it follows from
 Theorem \ref{theorem  gap  results of volume for hypersurface  non-constant S}.
 \end{proof}
 \begin{proof}[\textbf{Proof of   Theorem \ref{theorem  gap  results of volume for hypersurface  constant S f3 embedded}}]
% The proof is similar to the proof of  Theorem \ref{theorem  gap  results of volume for hypersurface  constant S f3}.
  By  the proof of Theorem \ref{theorem  gap  results of volume for hypersurface  constant S f3}, if $M^n$ has the constant squared length of second fundamental form  $S:=|A|^2>n$ and  constant third mean curvature $H_3$, then   $M^n$ is IE and 
 $$
  \int_{M}\varphi_a^2=\int_{M}\psi_a^2=\frac{1}{n+2}{\rm Vol }(M^n)
  \quad 
  \text{for all }  a \in \mathbb{S}^{n+1}.$$
  Substituting this into Theorem \ref{theorem  gap  results of volume in the m preimage varphi p2} proves the required gap in the case $S>n$.
This completes the proof by a similar discussion in the proof of Theorem \ref{theorem  gap  results of volume for hypersurface  constant S f3} if $0\leq S\leq n$.
%On the other hand, Simons' inequality \cite{Simons68} shows that if $0\leq S\leq n$, then either  $S\equiv0$ or $S\equiv n$  on $M^n$.
%The case of  $S\equiv n$   was characterized by Chern-do Carmo-Kobayashi \cite{Chern do Carmo Kobayashi 1970} and Lawson \cite{Lawson 1969}  independently:  the Clifford torus $M_{k,n-k}$ $\left(1 \leq k \leq n-1\right) $ are the only closed minimal hypersurfaces in $\mathbb{S}^{n+1}$  with    $S \equiv n$.
%It is easy to see that all the Clifford tori satisfy the required gap.
 \end{proof}

%Similar to Chern's conjecture,  a  natural problem is
%\begin{prob} \label{problem Similar to Chern's conjecture volume}
%Let $M^n$ be a  closed  immersed  minimal  hypersurface of the unit sphere $\mathbb{S}^{n+1}$ with constant scalar curvature. Then for each $n$, what is the set of all possible values for ${\rm Vol}(M^n)$?
%\end{prob}

%\section{Appendix. Some updated results }
%   During the two-years' review process of this paper, we have obtained several improvements. The first one further improves the following result by  Nguyen \cite{Nguyen  2023}  which generalized  the volume inequality of Li and Yau \cite{Li and Yau 1982}  by a method different from ours.\footnote{In October 11 of 2022, Nguyen sent an email to us asking where is the Yau conjecture proposed, and later he revised his arXiv paper to include this result.}   
%   \begin{thm}[Nguyen \cite{Nguyen  2023}]
%Let $f: M^n\looparrowright    \mathbb{S}^{N} $ be a closed minimal  immersed submanifold  in  $\mathbb{S}^{N}$. If there exists a point such that its preimage set  consists of $m$ distinct points in $M^n$, then
%$
 %                   {\rm Vol}(M^n)
 %                   \geq
 %                   m{\rm Vol}(\mathbb{S}^{n})$.
 %  \end{thm}
 %   In this appendix, we update our new %results.  
%\subsection{Main results}
% The first improvement shows that the volume lower bound $m{\rm Vol}(\mathbb{S}^{n})$ would not be attained once the minimal submanifold $M^n$ is not totally geodesic.

\section {Volume difference of minimal submanifold in two hemispheres}\label{sect5}
In this section, we present some results of volume difference for the two parts of minimal submanifold in the two hemispheres divided by an equator.
 \begin{lem}\label{lem int varphi a 2 leq varphi a}
   Let $M^n$ be a  closed minimal immersed  submanifold  in  $\mathbb{S}^{N}$. If $M\not\subset \left\lbrace\varphi_a=0\right\rbrace$, then
 	$$\int_{\{ \varphi_a \geq 0\}} \varphi_a^2\leq \frac{{\rm Vol}(\mathbb{B}^{n+1})}{2{\rm Vol}(\mathbb{B}^{n})} 
 	\int_{\{ \varphi_a \geq 0\}}  \varphi_a ,$$
 	    where the equality  holds if and only if $M^n$ is totally geodesic.
 	 \end{lem}
\begin{proof}
By the    co-area formula, (\ref{equation Stokess formula to varphia}) and Proposition \ref{prop The monotonicity formula for minimal submanifolds}, we obtain
$$
\begin{aligned}
\int_{\left\lbrace  \varphi_a \geq 0  \right\rbrace }
|a^{\rm T}|^2
&=
\int_0^1 dt
\int_{\left\lbrace \varphi_a=t\right\rbrace }
{|a^{\rm T}|}
=
\int_0^1 dt
\int_{\left\lbrace \varphi_a\geq t\right\rbrace }
n\varphi_a\\
& \leq \int_0^1 dt
\int_{\left\lbrace \varphi_a\geq 0\right\rbrace }
n\varphi_a(1-t^2)^{\frac{n}{2}}\\
&=\frac{{\rm Vol}(\mathbb{B}^{n+1})}{2{\rm Vol}(\mathbb{B}^{n})}
\int_{\left\lbrace \varphi_a\geq 0\right\rbrace }
n\varphi_a.
\end{aligned}
$$
By $\int_{\left\{\varphi_a \geq  0\right\}} \Delta \varphi_a^2=0$, we have
$$
\int_{\left\{ \varphi_a \geq 0   \right\}} n \varphi_a^2
=\int_{\left\lbrace  \varphi_a \geq 0 \right\rbrace }
|a^{\rm T}|^2
 \leq\frac{{\rm Vol}(\mathbb{B}^{n+1})}{2{\rm Vol}(\mathbb{B}^{n})}
 \int_{\left\lbrace \varphi_a\geq 0\right\rbrace }
 n\varphi_a.
$$
By a similar proof of Theorem \ref{thm  f_1(s)=1+s^2, f_2(s)=s, V(s)=1-s^2}, one has the equality above holds if and only if $M^n$ is totally geodesic.
\end{proof}

 \begin{thm}\label{theorem Volume difference}
 Let $M^n$ be a  closed minimal immersed  submanifold  in  $\mathbb{S}^{N}$. If $M\not\subset \left\lbrace\varphi_a=0\right\rbrace$, then
  \begin{itemize}
 	\item[(i)]
 $\left|  {\rm Vol} \{\varphi_a\geq 0\}  -
  {\rm Vol} \{\varphi_a\leq 0\}\right| 
\leq  {\rm Vol}(M^n)-(n+1) \int_M \varphi_a^2,$
 	\item[(ii)] 
 	$ {\rm Vol} \{\varphi_a\geq 0\} \geq  \frac{n+1}{2} \int_M \varphi_a^2$.
 	 \end{itemize}
  \end{thm}
  \begin{proof}
  By  $\int_{M}\Delta \varphi_{a}=0$,	Corollary \ref{thm  f_1(s)=1, f_2(s)=s, V(s)=1-s^2} and Lemma \ref{lem int varphi a 2 leq varphi a} , we have
  $$
  \int_{\{\varphi_a\geq 0\}}1\geq \frac{{\rm Vol}(\mathbb{S}^n)}{4{\rm Vol}(\mathbb{B}^n)}
  \cdot
 \frac{2{\rm Vol}(\mathbb{B}^{n})}{{\rm Vol}(\mathbb{B}^{n+1})} 
  \int_M \varphi_a^2
  =\frac{n+1}{2}\int_M \varphi_a^2.
  $$
By $\int_{\{\varphi_a\geq 0\}}1+\int_{\{\varphi_a\leq 0\}}1={\rm Vol}(M)$, one has
  $${\rm Vol}(M^n) \geq (n+1) \int_M \varphi_a^2+
  \left|  {\rm Vol} \{\varphi_a\geq 0\}  -
    {\rm Vol} \{\varphi_a\leq 0\}\right|. $$   
  \end{proof}
    By Theorem \ref{theorem  gap  results of volume in the m preimage varphi p2} and Theorem \ref{theorem Volume difference}, one has
     \begin{cor}\label{cor Volume gap of minimal hypersurface by Volume difference}
   Let $M^n$ be a  closed minimal immersed  submanifold  in  $\mathbb{S}^{N}$. If there exists a point $p\in f(M)$ such that its preimage set  consists of $m$ distinct points in $M^n$ and $M\not\subset \left\lbrace\varphi_p=0\right\rbrace$, then
  \begin{equation*} 
    {\rm Vol}(M^n) 
    \geq 
    m{\rm Vol}(\mathbb{S}^{n})+
     \frac{1}{n+2}
     \left|  {\rm Vol} \{\varphi_p\geq 0\}  -
    {\rm Vol} \{\varphi_p\leq 0\}\right| .
  \end{equation*}    
     \end{cor}
     In particular, we have the following corollaries for minimal hypersurfaces by Theorem \ref{theorem Volume difference}.

 \begin{cor}\label{cor Volume of minimal hypersurface partitioning}
 Let $M^n$ be a  closed, non-totally-geodesic, minimal immersed  hypersurface  in  $\mathbb{S}^{n+1}$. Then
  \begin{itemize}
 	\item[(i)]
 $\left|  {\rm Vol} \{\varphi_a\geq 0\}  -
  {\rm Vol} \{\varphi_a\leq 0\}\right| 
\leq  \int_M \psi_a^2,$
 	\item[(ii)] 
 	$ {\rm Vol} \{\varphi_a\geq 0\} \geq  \frac{n+1}{2} C(n, S) {\rm Vol}(M^n)$.  
 	\end{itemize}
 In particular, if $S \equiv$ Constant, then $C(n, S)=\frac{1}{2 n}$.
 \end{cor}
 \begin{cor}\label{cor Volume of minimal S constant IE hypersurface partitioning}
 Let $M^n$ be a  closed, non-totally-geodesic, minimal immersed  hypersurface with constant scalar curvature    in  $\mathbb{S}^{n+1}$. 
  If   the third mean curvature is constant (or $M^n$ is Integral-Einstein), then 
  $ {\rm Vol} \{\varphi_a\geq 0\} \geq  \frac{n+1}{2(n+2)}  {\rm Vol}(M^n)$. 
   \end{cor}
   \begin{rem}\label{rem Volume of minimal S constant IE hypersurface partitioning}
   In Corollary \ref{cor Volume of minimal S constant IE hypersurface partitioning},
  $$\lim_{n \to {+\infty }}
  \frac{{\rm Vol} \{\varphi_a\geq 0\} }{ {\rm Vol}(M^n)}=\frac{1}{2}   \quad 
    \text{for all }  a \in \mathbb{S}^{n+1}.$$
    This reflects that the volume of the minimal hypersurface with constant scalar curvature   and constant   third mean curvature (or $M^n$ is Integral-Einstein) in each hemisphere is almost the same    for sufficiently large $n$.
\end{rem}
Based on  Corollary \ref{cor Volume of minimal hypersurface partitioning}, Corollary \ref{cor Volume of minimal S constant IE hypersurface partitioning} and Remark \ref{rem Volume of minimal S constant IE hypersurface partitioning}, we propose the following conjectures.
\begin{conj}\label{conj  hypersurface 0.5 volume with constant scalar curvature}
 Let $M^n$ be a  closed  minimal immersed  hypersurface with constant scalar curvature  in  $\mathbb{S}^{n+1}$. If $M\not\subset \left\lbrace\varphi_a=0\right\rbrace$, then
  	$$ {\rm Vol} \{\varphi_a\geq 0\} = \frac{1}{2} {\rm Vol}(M^n).$$ 
\end{conj}
\begin{conj}\label{conj  hypersurface 0.5 volume}
 Let $M^n$ be a  closed  minimal immersed  hypersurface   in  $\mathbb{S}^{n+1}$. If $M\not\subset \left\lbrace\varphi_a=0\right\rbrace$, then
  	$$ {\rm Vol} \{\varphi_a\geq 0\} = \frac{1}{2} {\rm Vol}(M^n).$$ 
\end{conj}
\begin{rem}
Conjecture \ref{conj  hypersurface 0.5 volume with constant scalar curvature} and Conjecture \ref{conj  hypersurface 0.5 volume} are true for all isoparametric hypersurfaces in $\mathbb{S}^{n+1}$, since they are invariant under the antipodal map.
\end{rem}
     \begin{acknow}
     The authors are deeply grateful to the anonymous referee for their valuable comments and careful review. Their thoughtful feedback has been instrumental in strengthening the paper.
     %The authors thank the anonymous referee for their valuable suggestions.
     \end{acknow}
\singlespacing
{\noindent \bf  Data Availability Statement} This manuscript has no associated data.
    \singlespacing
{\noindent \bf  Declarations} 
    \singlespacing
    {\noindent\bf Conflict of interest}
On behalf of all authors, the corresponding author states that there is no Conflict of interest.


\begin{thebibliography}{99}
\bibitem{Andrews Ben Li Haizhong 2015}
B. Andrews and H. Z. Li, \emph{Embedded constant mean curvature tori in the three sphere},  J. Differ. Geom. \textbf{99} (2015),   169--189.

 \bibitem{Brendle S 2013}
S. Brendle,  \emph{Embedded minimal tori in $S^3$ and the Lawson conjecture}, Acta Math.  \textbf{211}  (2013), 177--190.

   \bibitem{Brendle S 2013 survey of recent results}
S. Brendle,   \emph{Minimal surfaces in $S^3$: a survey of recent results}, Bull. Math. Sci. \textbf{3} (2013), 133--171.

\bibitem{Brendle S 2020}
S. Brendle, \emph{Minimal hypersurfaces and geometric inequalities}, 
Ann. Fac. Sci. Toulouse Math.  \textbf{32} (2023),  179--201.
%arXiv:2010.03425v2.

\bibitem{Calabi 1967}
E. Calabi,  \emph{Minimal immersions of surfaces in Euclidean spheres}, J. Differ. Geom.  \textbf{1} (1967), 111--125. 


  \bibitem{Chang 1993}
 S. P. Chang, \emph{On minimal hypersurfaces with constant scalar curvatures in $\mathbb{S}^4$}, J. Differ. Geom.  \textbf{37} (1993), 523--534.


\bibitem{Cheng Qing Ming  Junqi Lai and Guoxin Wei 2024}
Q. M. Cheng, J. Q.  Lai and G. X. Wei, \emph{Area of Minimal Hypersurfaces in the Unit Sphere (II)},
J. Geom. Anal. \textbf{34} (2024),   Paper No. 369.

\bibitem{Cheng Qing Ming Guoxin Wei and Yuting Zeng 2019}
Q. M. Cheng, G. X. Wei and Y. T. Zeng, \emph{Area of minimal hypersurfaces in the unit sphere},
 Asian J. Math. \textbf{25} (2021),   183--194.

\bibitem{Cheng Li Yau 1984 Heat equations}
  S. Y. Cheng, P. Li and S. T. Yau, \emph{Heat equations on minimal submanifolds and their applications}, Amer. J. Math. \textbf{106} (1984), 1033--1065.

%\bibitem{Chern68}
%S. S. Chern, \emph{Minimal submanifolds in a Riemannian manifold}, Mimeographed Lecture
%Note, Univ. of Kansas, 1968.

 \bibitem{Chern do Carmo Kobayashi 1970}
  S. S. Chern, M. do Carmo, S. Kobayashi,  \emph{Minimal submanifolds of a sphere with second fundamental form of constant length},
 Functional Analysis and Related Fields, Springer-Verlag,  Berlin.
   (1970), 59--75.

\bibitem{Choe and Gulliver 1992}
  J. Choe and R. Gulliver,  \emph{Isoperimetric inequalities on minimal submanifolds of space forms},  Manuscripta
  Math. \textbf{77} (1992), 169--189.

\bibitem{Choi Wang 1983}
H. I. Choi and  A. N. Wang,  \emph{A first eigenvalue estimate for minimal hypersurfaces}, J. Differ. Geom. \textbf{18} (1983), 559--562.


  \bibitem{Qi Ding  Y.L. Xin. 2011}
 Q. Ding and  Y. L. Xin,  \emph{On Chern's problem for rigidity of minimal hypersurfaces in the spheres},  Adv. Math. \textbf{227} (2011), 131--145.
% \bibitem{Qi Ding  Y.L. Xin. 2011}
 %Q. Ding and  Y. L. Xin,  \emph{On Chern's problem for rigidity of minimal hypersurfaces in the spheres},  Adv. Math. \textbf{227} (2011), 131--145.

\bibitem{Frankel 1966}
 T.  Frankel,
 \emph{On the fundamental group of a compact minimal submanifold},
  Ann.  Math.  \textbf{83} (1966), 68--73.

  \bibitem{Ge19}
  J. Q. Ge, \emph{Problems related to isoparametric theory}, In: Surveys in Geometric Analysis 2019, pp. 71--85, Ed. by: G. Tian,
 Q. Han and Z. L. Zhang, Science Press Beijing, 2020.

\bibitem{Ge Li 2020}
J. Q. Ge and F. G. Li, \emph{Integral-Einstein hypersurfaces in spheres}, Comm. Anal. Geom. 33 (2025),   1905--1930.

\bibitem{Ge Li 2022 Solomon-Yau conj}
   J. Q. Ge and F. G. Li,  \emph{Volume gap for minimal submanifolds in spheres}, arXiv:2210.04654.
   
\bibitem{Ge Li 2021 Perdomo conjecture}
   J. Q. Ge and F. G. Li,  \emph{A lower bound for $L_2$ length of second fundamental form on minimal hypersurfaces},     {Proc. Amer. Math. Soc.}       \textbf{150} (2022), 2671--2684.


\bibitem{Alfred Gray 1973}
A. Gray. \emph{The volume of a small geodesic ball of a Riemannian manifold},  Michigan Math. J. \textbf{20} (1973), 329--344.

\bibitem{Hsiang 1967}
W. Y. Hsiang,
\emph{Remarks on closed minimal submanifolds in the standard Riemannian $m$-sphere},
 J. Differ.  Geom. \textbf{1} (1967), 257--267.

 \bibitem {IW15}
T. Ilmanen and B. White, \emph{Sharp lower bounds on density for area-minimizing cones}, Camb. J. Math.  \textbf{3} (2015), 1--18.

\bibitem{Lawson 1969}
 H. B. Lawson, \emph{Local Rigidity Theorems for Minimal Hypersurfaces},  Ann. Math. \textbf{89} (1969), 187--197.


    \bibitem{Li Lei Hongwei Xu  Zhiyuan Xu 2021}
L. Li, H. W. Xu and Z. Y. Xu,
 \emph{On the generalized Chern conjecture for hypersurfaces with constant mean curvature in a sphere},
 Sci. China Math. \textbf{64} (2021),  1493--1504.

 \bibitem {Li and Yau 1982}
P. Li and  S. T. Yau, \emph{
A new conformal invariant and its applications to the Willmore conjecture and the first eigenvalue of compact surfaces},
Invent. Math. \textbf{69} (1982), 269--291.


\bibitem {Marques and Neves 2014}
F. C. Marques and A.  Neves, \emph{Min-max theory and the Willmore conjecture}, Ann.  Math. \textbf{179}  (2014), 683--782.

\bibitem {Nguyen  2023}
 M. T.  Nguyen,
\emph{Weighted monotonicity theorems and applications to minimal surfaces of $\mathbb{H}^n$ and $\mathbb{S}^n$},  
Trans. Amer. Math. Soc. \textbf{376} (2023),   5899--5921. 

 \bibitem{Nomizu and Smyth 1969}
 K. Nomizu and B. Smyth, \emph{On the Gauss Mapping for Hypersurfaces of Constant Mean Curvature in the Sphere}, Comm. Math. Helv. \textbf{44} (1969), 484--490.

\bibitem{Osserman Robert 1980}
 R.  Osserman,\emph{Minimal surfaces, Gauss maps, total curvature, eigenvalue estimates, and stability.} The Chern Symposium 1979 (Proc. Internat. Sympos., Berkeley, Calif., 1979), pp. 199--227, Springer, New York-Berlin, 1980.


  \bibitem{Otsuki 1970}
  T. Otsuki, \emph{Minimal hypersurfaces in a Riemannian manifold of constant curvature},  Amer. J.
  Math. \textbf{92} (1970), 145--173.


%\bibitem{Peng and Terng 1983}
%C. K. Peng and C. L. Terng, \emph{Minimal hypersurfaces of spheres with constant scalar curvature}, In: Seminar on Minimal Submanifolds, pp. 177--198, Ann. Math. Stud., Princeton Univ. Press,  Princeton, NJ, 1983.

\bibitem{Peng and Terng1 1983}
C. K. Peng and C. L. Terng, \emph{The scalar curvature of minimal hypersurfaces in spheres},
Math. Ann. \textbf{266} (1983), 105--113.

\bibitem{Perdomo and  Wei 2015}
O. Perdomo and G. Wei, \emph{$n$-dimensional area of minimal rotational hypersurfaces in spheres},
Nonlinear Anal. \textbf{125} (2015), 241--250.

  \bibitem{M Scherfner S Weiss and  Yau 2012}
  M. Scherfner, S. Weiss and S. T. Yau, \emph{A review of the Chern conjecture for isoparametric hypersurfaces in spheres}, In: Advances in Geometric Analysis, pp. 175--187, Adv. Lect. Math. (ALM), \textbf{21}, Int. Press, Somerville, MA, 2012.

  %\bibitem{Schoen Richard and S. T. Yau 1994}
%  R. Schoen, and S. T. Yau, \emph{Lectures on Differential Geometry},  International Press, 1994.

 \bibitem{Simons68}
    J. Simons, \emph{Minimal varieties in Riemannian manifolds}, Ann. Math. \textbf{88} (1968), 62--105.

\bibitem{Takahashi 1966}
T. Takahashi, \emph{Minimal immersion of Riemannian manifolds}, J. Math. Soc. Japan \textbf{18} (1966), 380--385.

% \bibitem{TWY18}
%Z. Z. Tang, D. Y. Wei and W. J. Yan, \emph{A sufficient condition for a hypersurface to be isoparametric,}  Tohoku Math. J.  \textbf{72} (2020), 493--505.

 \bibitem{TY20}
 Z. Z. Tang and W. J. Yan, \emph{On the Chern conjecture for isoparametric hypersurfaces},  Sci. China Math. \textbf{66} (2023),  143--162.
 %Z. Z. Tang and W. J. Yan, \emph{On the Chern conjecture for isoparametric hypersurfaces},  to appear in Sci. China Math., arXiv:2001.10134.

 \bibitem{Urakawa Hajime Spectral geometry of the Laplacian2017}
 H. Urakawa, \emph{Spectral geometry of the Laplacian. Spectral analysis and differential geometry of the Laplacian.} 
 World Scientific Publishing Co. Pte. Ltd., Hackensack, NJ, 2017. xii+297 pp. ISBN: 978-981-3109-08-7 
%\bibitem{Xu  H.  Xu 2017}
 %H. W. Xu  and  Z. Y. Xu, \emph{On Chern's conjecture for minimal hypersurfaces and rigidity of self-shrinkers}, J. Funct. Anal. \textbf{273} (2017), 3406--3425.

% \bibitem{Cheng Qing Ming and H Yang 1998}
  %  H. C. Yang and Q. M. Cheng, \emph{Chern's conjecture on minimal hypersurfaces}. Math. Z. \textbf{227} (1998), 377--390.

% \bibitem{Yau  1982}
% S. T. Yau,  \emph{Problem section}, In: Seminar on Differential Geometry, pp. 669--706, Ann. Math. Stud., \textbf{102}, Princeton Univ. Press, Princeton, NJ, 1982.

\bibitem{Celso Viana 2023}
C.  Viana,  \emph{Isoperimetry and volume preserving stability in real projective spaces,}
J. Differential Geom. \textbf{125} (2023), 187--205. 

\bibitem{Yau  1992}
 S. T. Yau, \emph{Chern-A great geometer of the twentieth century}, International Press Co. Ltd. Hong Kong, 1992.
\end{thebibliography}
\end{document}